\documentclass[10pt,a4paper,draft]{article}

\usepackage{amsmath,amsthm,amsopn,amsfonts,amssymb,dsfont,color}
\usepackage[dvips]{graphicx}
\usepackage{psfrag}
\usepackage{enumerate}

\usepackage{a4,a4wide}

\def\ep{{\varepsilon}}
\def\R{\mathbb R}

\newtheorem{theo}{\textbf{Theorem}}[section]

\newtheorem{lem}[theo]{\textbf{Lemma}}
\newtheorem{prop}[theo]{\textbf{Proposition}}

\newtheorem{defi}[theo]{\textbf{Definition}}

\newtheorem{assumption}[theo]{\textbf{Assumption}}
\newtheorem{rem}[theo]{\textbf{Remark}}

\title{Varying the direction of propagation in
 reaction-diffusion equations in periodic
media}
\date{}

\begin{document}

\maketitle

\begin{center}
{\large\bf Matthieu Alfaro \footnote{ I3M, Universit\'e de
Montpellier 2, CC051, Place Eug\`ene Bataillon, 34095 Montpellier
Cedex 5, France. E-mail: matthieu.alfaro@univ-montp2.fr}
 and  Thomas Giletti
\footnote{IECL, Universit\'{e} de Lorraine, B.P. 70239, 54506
Vandoeuvre-l\`{e}s-Nancy Cedex, France. E-mail:
thomas.giletti@univ-lorraine.fr}.}\\
[2ex]

\end{center}

\vspace{15pt}

\tableofcontents

\vspace{10pt}

\begin{abstract} We consider a multidimensional reaction-diffusion equation of either ignition or monostable type, involving periodic heterogeneity, and analyze the dependence of the
propagation phenomena on the direction. We prove that the
(minimal) speed of the underlying pulsating fronts depends
continuously on the direction of propagation, and so does its
associated profile provided it is unique up to time shifts. We
also prove that the spreading properties \cite{Wein02} are
actually uniform with respect to the
direction.\\


\end{abstract}

\section{Introduction}\label{s:intro}

In this work, we focus on the  heterogeneous reaction-diffusion
equation
\begin{equation}\label{monostable}
\partial _t u=  \mbox{div} (A (x) \nabla u)  + q (x) \cdot \nabla u + f(
x,u), \quad t\in \R, x\in \R ^N.
\end{equation}
Here $A = (A_{i,j})_{1 \leq i,j \leq N}$ is a matrix field, and $q = (q_1 , ... ,q_N)$
is a vector
 field, to be precised later. The nonlinearity $f$ is of either the
monostable type (not necessarily with the KPP assumption) or
ignition type, which we will define below. We would like to
understand, in the periodic framework, how the propagation
phenomena depend on the direction.

On the one hand, we prove that the minimal (and, in the ignition
case, unique) speed of the well known pulsating fronts depends
continuously on the direction of propagation. On the other hand,
we prove that the spreading properties are in some sense uniform
with respect to the direction, thus improving the seminal result
of Weinberger \cite{Wein02}. While in the ignition case, these
properties will mostly follow from the well known uniqueness of
the pulsating traveling wave, such uniqueness does not hold true
in the monostable case where the set of admissible speeds is
infinite. Our argument will be inspired
by~\cite{berestycki-nirenberg}, \cite{Ber-Ham-02}, and will rely
on an approximation of the monostable nonlinearity by some
well-chosen ignition nonlinearities.

\subsection{Main assumptions}

Let $L_1$,...,$L_N$ be given positive constants. A function $h:\R
^ N \to \R$ is said to be {\it periodic} if
$$
h(x_1,...,x_k+L_k,...,x_N)=h(x_1,...,x_N),
$$
for all $1\leq k \leq N$, all $(x_1,...,x_N)\in \R^N$. In such
case, $\mathcal{C} = (0,L_1)\times\cdots \times (0,L_N)$ is called
the cell of periodicity. Through this work, we put ourselves in
the spatially periodic framework and assume that
\begin{equation}\label{periodicity}
\begin{array}{l}
\mbox{ for all } 0 \leq i, j \leq N, \, \mbox{ the functions } A_{i,j}: \R^N \to \R \, ,  \ q_i : \R^N \to \R \quad \mbox{ are periodic},\vspace{3pt}\\
\text{ for all } u\in \R_+,\, \mbox{ the function } f(\cdot,u):
\R^N \to \R \text{ is periodic}.
\end{array}
\end{equation}
Moreover, we assume that $A = (A_{i,j})_{1 \leq i,j \leq N}$ is a $C^3$ matrix field which satisfies
\begin{equation}\label{ellipticity}
\begin{array}{l}
A(x) \mbox{ is a symmetric matrix for any } x \in \R^N,\vspace{3pt}\\
\exists 0 < a_1 \leq a_2 < \infty, \quad \forall (x, \xi) \in \R^N
\times \R^N, \quad a_1 |\xi|^2 \leq \sum_{i,j} A_{i,j} (x) \xi_i
\xi_j \leq a_2 |\xi |^2.
\end{array}
\end{equation}
Concerning the advection term, we assume that $q = (q_1 , ... ,q_N)$
is a $C^{1,\delta}$ vector
 field, for some $\delta >0$, which satisfies
 \begin{equation}
\begin{array}{l}
\mbox{div} \, q =0 \mbox{ in } \R^N \quad \mbox{ and } \quad
\forall 0 \leq i \leq N , \ \displaystyle \int_{\mathcal{C}} q_i =
0.
\end{array}
\end{equation}
The  advection term in the equation is mostly motivated by
combustion
 models where the dynamics of the medium also plays an essential role. In such
 a context, the fact that the flow $q$ has zero divergence carries the physical
  meaning that the medium is incompressible.

Furthermore, we will assume that $f$ satisfies either of the
following two assumptions.
\begin{assumption}[Monostable nonlinearity]\label{hyp:monostable} The function $f:\R ^N \times \R_+\to \R$ is of class $C^{1,\alpha}$ in $(x,u)$
and~$ C^2$ in $u$, and nonnegative on $\R^N\times [0,1]$. Concerning the steady states of the periodic
equation~\eqref{monostable}, we assume that
\begin{enumerate}[(i)]
\item the constants 0 and 1 are steady states (that is,
$f(\cdot,0)\equiv f (\cdot,1) \equiv
0$ in $\R^N$); 
\item $\forall u \in (0,1), \ \exists x \in \R^N, \ \  f(x,u) >0$.
\item there exists some $\rho
>0$ such that $f (x,u)$ is nonincreasing with respect to $u$ in
the set $\R^N \times (1-\rho,1]$.
\end{enumerate}
\end{assumption}

Notice that, if $0\leq p(x)\leq 1$ is a periodic stationary state, then $p\equiv 0$ or $p\equiv 1$. Indeed, since $f(x,p)\geq 0$, the strong maximum principle enforces $p$ to be identically equal to its minimum, thus constant and, by $(ii)$, the constant has to be 0 or 1.

\begin{assumption}[Ignition nonlinearity]\label{hyp:ignition} The function $f:\R ^N \times \R_+\to \R$ is locally Lipschitz-continuous on $\R^N \times \R_+$.
Concerning the steady states of the periodic
equation~\eqref{monostable}, we assume that
\begin{enumerate}[(i)]
\item there exists $0<\theta<1 $ such that
$$\forall 0 \leq u \leq \theta, \ \ \forall x \in \R^N, \qquad f(x,u)=0,$$
as well as
$$\forall x \in \R^N , \qquad f(x,1)=0.$$
\item $\forall u \in (\theta,1), \ \exists x \in \R^N, \ \  f(x,u)
>0$.
\item there exists some
$\rho
>0$ such that $f (x,u)$ is nonincreasing with respect to $u$ in
the set $\R^N \times (1-\rho,1]$.
\end{enumerate}
\end{assumption}

Notice that, similarly as above, $(ii)$ implies that if $\theta\leq p(x)\leq 1$ is a periodic stationary state then $p\equiv \theta$ or $p\equiv 1$.

\subsection{Comments and related results}

Under Assumption~\ref{hyp:monostable},
Assumption~\ref{hyp:ignition}, equation~\eqref{monostable} is
referred to as the monostable equation, the ignition equation
respectively. Both sets of assumptions arise in various fields of
physics and the life sciences, and especially in combustion and
population dynamics models where propagation phenomena are
involved. Indeed, a particular feature of these equations is the
formation of traveling fronts, that is particular  solutions
describing the transition at a  constant speed from one stationary
solution to another. Such solutions have proved in numerous
situations their utility in describing the spatio-temporal
dynamics of a population, or the propagation of a flame modelled
by a reaction-diffusion equation.

Equation \eqref{monostable} is a heterogeneous version of the well
known reaction-diffusion equation
\begin{equation}\label{homogene}
\partial _t u=\Delta u +f(u),
\end{equation}
where typically $f$ belongs to one of the three following classes:
monostable, ignition and bistable.  Homogeneous reaction-diffusion
equations have been extensively studied in the literature (see
\cite{Kan1}, \cite{Aro-Wei1, Aro-Wei2}, \cite{Fif-Mac},
\cite{Ber-Nic-Sch}, \cite{Volpert-Volpert-Volpert} among others)
and are  known to support the existence of monotone traveling
fronts. In particular, for monostable nonlinearities, there exists
a critical speed $c^*$ such that all speeds $c\geq c^*$ are
admissible, while in the bistable and ignition cases, the
admissible speed $c= c^*$ is unique. Moreover, in both cases, the
speed $c^*$ corresponds to the so-called  spreading speed of
propagation of compactly supported initial data.

Among monostable nonlinearities, one can distinguish the ones
satisfying the Fisher-KPP assumption, namely $u\mapsto
\frac{f(u)}u$ is maximal at 0 (meaning that the growth per capita
is maximal at small densities), the most famous example being
introduced by Fisher \cite{Fish} and Kolmogorov, Petrovsky and
Piskunov \cite{Kol-Pet-Pis} to model the spreading of advantageous
genetic features in a population:
\begin{equation*}
 \partial_t u=\Delta u+u(1-u).
\end{equation*}
Let us notice that our work stands in the larger class of
monostable nonlinearities.

Nevertheless, much attention was more recently devoted to the
introduction of some heterogeneity, taking various forms such as
advection, spatially dependent diffusion or reaction term. Taking
such a matter into account is essential as far as models are
concerned, the environment being rarely homogeneous and may depend
in a non trivial way on the position in space (patches, periodic
media, or more general heterogeneity...). We refer to the seminal
book of Shigesada and Kawasaki \cite{Shi-Kaw}, and the
enlightening introduction in \cite{Ber-Ham-Roq1} where the reader
can find very precise and various references. As far as combustion models are concerned, one can consult  \cite{Ber-Nic-Sch}, \cite{Xin} and the references therein.

Traveling front solutions in heterogeneous versions of
\eqref{homogene} with periodicity in space, in time, or more
general media are studied in \cite{Wein02},
\cite{berestycki-nirenberg}, \cite{HZ}, \cite{Xin3},
\cite{Ber-Ham-02}, \cite{Ber-Ham-Roq2}, \cite{Nad-09},
\cite{Nolen-Ryzhik} among others. For very general
reaction-diffusion equations, we refer to
\cite{berestycki-hamel-cpam2} for a definition of generalized
transition waves and their properties.

In this work, we restrict ourselves to the spatially periodic
case, which provides insightful information on the role and
influence of the heterogeneity on the propagation, as well as a
slightly more common mathematical framework. In this periodic
setting, let us mention the following keystone results for
ignition and monostable nonlinearities. Weinberger \cite{Wein02}
exhibited a direction dependent spreading speed for planar-shaped
initial data and proved, in the monostable case, that this
spreading speed is also the minimal speed of pulsating traveling
waves moving in the same direction. His approach relies on a
discrete formalism, in contrast with the construction of both
monostable and ignition pulsating traveling waves by Berestycki
and Hamel \cite{Ber-Ham-02}, via more flexible PDE technics. In
this PDE framework, note also the work of Berestycki, Hamel and
Roques \cite{Ber-Ham-Roq2} where KPP pulsating fronts are
constructed without assuming the nonnegativity of the
nonlinearity. Our main goal is to study how these results behave
when we vary the direction of propagation.

\medskip

Let us give another motivation for our analysis of the dependence
of the propagation on the direction. Our primary interest was
actually to study the {\it sharp interface limit} $\ep \to 0$ of
\begin{equation}\label{papier2}
 \partial _t u^\ep= \ep \Delta u^\ep+\displaystyle \frac 1 \ep f\left(\frac x
 \ep,u^\ep\right),
\end{equation}
arising  from the hyperbolic space-time rescaling
$u^\ep(t,x):=u\left(\frac t \ep, \frac x \ep\right)$ of
\eqref{monostable}, with $A\equiv Id$, $q\equiv 0$. The parameter
$\ep>0$ measures the thickness of the diffuse interfacial layer.
As this thickness tends to zero, \eqref{papier2} converges --- in
some sense--- to a limit interface, whose motion is governed by
the minimal speed (in each direction) of the underlying pulsating
fronts. This dependence of the speed on the (moving) normal
direction is in contrast with the homogeneous case and makes the
analysis quite involved. In particular, it turns out that we need
to improve (by studying the uniformity with respect to the
direction) the known spreading properties \cite{Wein02},
\cite{Ber-Ham-02}, for both ignition and monostable nonlinearities
in periodic media. We refer to \cite{A-Gil} for this singular
limit analysis, using some of the results of the present work.

\section{Main results}\label{s:results}

Before stating our main results in subsection~\ref{ss:varying},
let us recall the classical results on both pulsating fronts and
spreading properties in subsection~\ref{ss:known}.

\subsection{Pulsating fronts and spreading properties: known results}\label{ss:known}

The definition of the so-called pulsating traveling wave was
introduced by Xin \cite{Xin} in the framework of flame
propagation. It is the natural extension, in the periodic
framework, of classical traveling waves. Due to the interest of
taking into account the role of the heterogeneity of the medium on
the propagation of solutions, a lot of attention was later drawn
on this subject. As far as monostable and ignition pulsating
fronts are concerned, we refer to the seminal works of Weinberger
\cite{Wein02}, Berestycki and Hamel \cite{Ber-Ham-02}. Let us also
mention \cite{Ber-Ham-Roq2}, \cite{Ham}, \cite{Ham-Roq},
\cite{Nad-09} for related results.

\medskip

 For the sake of completeness,
let us first recall the definition of a pulsating traveling wave
for the equation~\eqref{monostable}, as stated in
\cite{Ber-Ham-02}.

\begin{defi}[Pulsating traveling wave]\label{def:puls}
A pulsating traveling wave solution, with speed $c >0$ in the
direction $n \in \mathbb{S}^{N-1}$, is an entire solution $u(t,x)$
--- $t\in\R$, $x\in \R ^N$--- of \eqref{monostable} satisfying
$$
\forall k \in \prod_{i=1}^N L_i \mathbb{Z} , \qquad u(t
,x)=u\left(t + \frac{k\cdot n}{c} ,x+k\right),
$$
for any $t \in \R$ and $x\in \R^N$, along with the asymptotics
$$
u(-\infty,\cdot)=0  < u (\cdot , \cdot) < u(+\infty,\cdot) = 1,
$$
where the convergences in $\pm \infty$ are understood to hold
locally uniformly in the space variable.
\end{defi}

One can easily check that, for any $c>0$ and $n \in
\mathbb{S}^{N-1}$, $u(t,x)$ is a pulsating traveling wave with
speed $c$ in the direction $n$ if and only if it can be written in
the form
$$
u(t,x) = U (x\cdot n -ct,x),
$$
where~$U(z,x)$ --- $z\in \R$, $x\in \R^N$--- satisfies
$$
\text{ for all } z\in \R,\, U(z,\cdot): \R^N \to \R \text{ is
periodic},
$$
\begin{equation*}
 U (-\infty, \cdot) = 1  < U (\cdot , \cdot) <   U (+\infty,\cdot)
 =0 \quad \text{ uniformly w.r.t. the space variable},
\end{equation*}
along with the following equation
\begin{equation}\label{eq-tw}
\begin{array}{l}
\mbox{div}_x (A \nabla_x U) + (n \cdot An)\, \partial_{zz} U + \mbox{div}_x ( An \,  \partial_z U) + \partial_z (n\cdot A \nabla_x U) \vspace{3pt}\\
\qquad \qquad + \, q \cdot \nabla_x U +  (q \cdot n) \, \partial_z
U   + c\partial _z U+f(x,U)=0 , \quad \text{ on } \R\times \R ^N.
\end{array}
\end{equation}

We can now recall the results of \cite{Ber-Ham-02} (see also
\cite{Wein02} for the monostable case), on the existence of
pulsating traveling waves for the spatially periodic monostable
and ignition equations. Precisely, the following holds.

\begin{theo}[Monostable and ignition pulsating fronts, \cite{Ber-Ham-02},\cite{Wein02}]

\begin{itemize}
\item Assume that $f$ is of the spatially periodic monostable
type, i.e. $f$ satisfies \eqref{periodicity} and
Assumption~\ref{hyp:monostable}. Then for any $n \in
\mathbb{S}^{N-1}$, there exists $c^* (n) >0$ such that pulsating
traveling waves with speed $c$ in the direction $n$ exist if and
only if $c \geq c^* (n).$ \item Assume that $f$ is of the
spatially periodic ignition type, i.e. $f$ satisfies
\eqref{periodicity} and Assumption~\ref{hyp:ignition}. Then for
any $n \in \mathbb{S}^{N-1}$, there exists a unique (up to time
shift) pulsating traveling wave, whose speed we denote by $c^* (n)
>0$.\medskip
\end{itemize}
Furthermore, in both cases, any pulsating traveling wave is
increasing in time.
\end{theo}

The introduction of these pulsating traveling waves was motivated
by their expected role in describing the large time behavior of
solutions of \eqref{monostable} for a large class of initial data.
In this context, let us state the result of \cite{Wein02} for
planar-shaped initial data.

\begin{theo}[Spreading properties,
\cite{Wein02}]\label{Wein02_cauchy} Assume that $f$ is of the
spatially periodic monostable or ignition type, i.e. $f$ satisfies
\eqref{periodicity} and either of the two
Assumptions~\ref{hyp:monostable} and~\ref{hyp:ignition}. Let $u_0$
be a nonnegative and bounded initial datum such that $\| u _0
\|_\infty < 1$ and$$ \exists C>0, \qquad x \cdot n \geq C \
\Rightarrow \ u_0 (x) = 0,
$$
$$
\liminf_{x \cdot n \to -\infty} u_0 (x) >0 \ \ (monostable \ case)
, \qquad \liminf_{x \cdot n \to -\infty} u_0 (x) >\theta \ \
(ignition \ case)
$$
for some $n\in \mathbb S ^{N-1}$.

Then the solution $u$ of \eqref{monostable}, with initial datum
$u_0$, spreads with speed $c^* (n)$ in the $n$-direction in the
sense that
\begin{equation}\label{conclusion_wein2}
\forall c <c^*(n), \qquad \displaystyle \lim_{t \to +\infty} \sup_{x \cdot n \leq ct} |
1 - u (t,x) | =0,
\end{equation}
\begin{equation}\label{conclusion_wein1}
\forall c > c^*(n) , \qquad \lim_{t \to +\infty} \sup_{x \cdot n \geq  c t }
u (t,x) =0.
\end{equation}
\end{theo}

\begin{rem}[Link between spreading speed and wave speed]\label{rem_bug} 
In \cite{Wein02}, Weinberger was actually concerned with a more
general discrete formalism
 where pulsating waves are not always known to exist. Therefore, the fact that the
 spreading speed and the minimal wave speed are one and the same was only explicitly
  stated in the monostable case.
  
  However, under the ignition Assumption~\ref{hyp:ignition}
  and benefiting from the results in~\cite{Ber-Ham-02}, it is clear by a simple comparison
  argument that the solution associated with any such initial datum spreads at most with the wave speed $c^* (n)$, namely \eqref{conclusion_wein1} holds true. Furthermore, one may check, using for instance $U^* (x \cdot n -
(c^*(n)-\alpha)t,x) - \delta$ as a subsolution of
\eqref{monostable}, where $U^*$ is the pulsating wave with speed
$c^* (n)$ and $\alpha >0$, $\delta >0$ are small enough, that
\eqref{conclusion_wein2} also holds true, at least for some large
enough initial data. Thus, the spreading speed exhibited by
Weinberger must be $c^* (n)$, as one would expect.

We will use a very similar argument in
Section~\ref{s:spread_ignition}, which is why we omit the details.
Moreover, it is a simplification of a
 classical argument, which originates from \cite{Fif-Mac} in the homogeneous framework, and usually aims at proving the stronger
 property that the profile of such a solution $u(t,x)$ of the Cauchy problem converges to that of the
 ignition pulsating wave. We refer for instance to the work of Zlato{\v{s}}~\cite{Zlatos}, which
 dealt with a fairly general multidimensional heterogeneous (not necessarily periodic) framework,
 and covers the above result under the additional assumption that $f(x,u)$ is bounded from below by a standard homogeneous ignition nonlinearity.
\end{rem}

Various features of pulsating fronts and many generalizations of
spreading properties have been studied recently. Nevertheless, as
far as we know, nothing is known on the dependence of these
results on the direction of propagation. Our results stand in this
new framework and are stated in the next subsection.

\subsection{Pulsating fronts and spreading properties: varying the direction}\label{ss:varying}

As recalled above, the periodic ignition equation admits a unique
pulsating traveling wave in any direction $n \in
\mathbb{S}^{N-1}$, while the periodic monostable
equation~\eqref{monostable} admits pulsating traveling waves in
any direction $n \in \mathbb{S}^{N-1}$, for any speed larger than
some critical $c^* (n) >0$. The
 latter is a consequence of the former, as was proved in \cite{Ber-Ham-02} by approximating the monostable equation with an ignition type equation. With some
 modifications of their argument, we will prove the following continuity property.

\begin{theo}[Continuity of minimal
speeds]\label{th:continuity_speeds} Assume that $f$ is of the
spatially periodic monostable or ignition type, i.e. $f$ satisfies
\eqref{periodicity} and either of the two
Assumptions~\ref{hyp:monostable} or~\ref{hyp:ignition}.

Then the mapping $n \in \mathbb{S}^{N-1} \mapsto c^* (n)$ is
continuous.
\end{theo}

In the Fisher-KPP case the continuity of the velocity map
$n\mapsto c^*(n)$, even if not explicitly stated, seems to follow
from the characterization of $c^*(n)$ (see \cite{Wein02},
\cite{Ber-Ham-02}). However, for other types of nonlinearities (and in particular, in the more general monostable case), such a property seems to be far from obvious.

For the sake of completeness, let us state the continuity of the
profile of the ignition wave, which will be proved simultaneously.

\begin{theo}[Continuity of ignition waves]\label{th:continuity_igwaves}
If $f$ satisfies \eqref{periodicity} and the ignition
Assumption~\ref{hyp:ignition}, then the mapping
$$n \in \mathbb{S}^{N-1} \mapsto U^* (z,x;n)$$
is continuous with respect to the uniform topology, where
$$u^* (t,x;n) = U^* (x \cdot n  - c^* (n) t ,x;n)$$
is the unique pulsating traveling wave in the $n$ direction,
normalized by $\min _{x\in \R^N}U^* (0,x;n)= \frac{1+ \theta
}{2}$.
\end{theo}

In Section~\ref{s:ignition}, we deal with the ignition case,
proving both the continuity of the speed (Theorem
\ref{th:continuity_speeds}) and that of the profile (Theorem
\ref{th:continuity_igwaves}). To do so we take advantage of the
uniqueness
 of the pulsating wave in each direction.

 Then, in Section \ref{s:monostable-speed}, we
approach our original monostable equation by some ignition type
problems, and prove that the associated ignition speeds converge
to $c^* (n)$ not only pointwise (as in \cite{Ber-Ham-02}), but
even uniformly with respect to $n \in \mathbb{S}^{N-1}$. The
continuity of the minimal speed (Theorem
\ref{th:continuity_speeds}) then immediately follows.
Unfortunately, the lack of a rigorous uniqueness result of the
monostable pulsating wave with minimal speed (at least up to our
knowledge) prevents us from stating
 continuity of its profile with respect to the speed of propagation. We refer to~\cite{Ham-Roq} for uniqueness
 results in the Fisher-KPP case and discussion on the general monostable framework.

\medskip

We also stated above the well known fact that for any planar-like
initial data in some direction~$n$, the associated solution of
\eqref{monostable} spreads in the $n$ direction with speed $c^*
(n)$. Our main result consists in improving (compare
Theorem~\ref{th:unif_spreading} with Theorem~\ref{Wein02_cauchy})
this property by adding some uniformity with respect to $n \in
\mathbb{S}^{N-1}$, as follows.

\begin{theo}[Uniform spreading]\label{th:unif_spreading}
Assume that $f$ is of the spatially periodic monostable or
ignition type, i.e. $f$ satisfies \eqref{periodicity} and either
Assumption~\ref{hyp:monostable} or Assumption~\ref{hyp:ignition}.
Let a family of nonnegative initial data $(u_{0,n})_{n \in
\mathbb{S}^{N-1}}$ be such that
\begin{equation}\label{hyp1}
\exists C >0, \qquad \forall n \in \mathbb{S}^{N-1}, \qquad x
\cdot n \geq C \ \Rightarrow \ u_{0,n} (x) = 0,
\end{equation}
\begin{equation}\label{hyp2}
\left.
\begin{array}{l}
\exists \mu > \theta  \ \ (ignition \ case) \vspace{3pt}\\
\exists \mu > 0  \ \ (monostable \ case)
\end{array}
\right\} , \qquad \exists K >0, \qquad \inf_{n \in
\mathbb{S}^{N-1}, \; x \cdot n \leq -K } u_{0,n} (x) \geq \mu,
\end{equation}
\begin{equation}\label{hyp0}
\sup_{n \in \mathbb{S}^{N-1}} \sup_{x \in \R^N}
 u_{0,n} (x)
 < 1.
\end{equation}
We denote by $(u_n)_{n \in \mathbb{S}^{N-1}}$ the associated
family of solutions of \eqref{monostable}.

Then, for any $\alpha >0$ and $\delta >0$, there exists $\tau
>0$ such that for all $t \geq \tau$,
\begin{equation}\label{conclusion1}
\sup_{n \in \mathbb{S}^{N-1}} \ \sup_{x \cdot n \leq (c^*(n) -
\alpha)t} |1 -u_n (t,x)| \leq  \delta,
\end{equation}
\begin{equation}\label{conclusion2}
\sup_{n \in \mathbb{S}^{N-1}} \ \sup_{x \cdot n \geq (c^*(n) +
\alpha)t} u_n (t,x) \leq  \delta.
\end{equation}
\end{theo}

The difficult part is again to deal with the monostable case. The
proof of the lower spreading property \eqref{conclusion1} will
again rely on an ignition approximation of the monostable
equation, whose traveling waves will serve as nontrivial
subsolutions of \eqref{monostable}. This is performed in
Section~\ref{s:lower-spreading}. Then,
Section~\ref{s:upper-spreading} is devoted to the proof of the
upper spreading property: we prove \eqref{conclusion2} in
subsection~\ref{ss:with} and, for sake of completeness, relax
assumption~\eqref{hyp0} in subsection~\ref{ss:without}.

Last, in Section \ref{s:spread_ignition}, we prove Theorem
\ref{th:unif_spreading} in the ignition case.

\section{Continuity of ignition waves}\label{s:ignition}

Let us here consider a periodic nonlinearity $f$ of the ignition
type, namely satisfying Assumption~\ref{hyp:ignition}. As
announced, we will prove simultaneously the continuity of both
mappings $n  \mapsto c^* (n)$ and $n \mapsto U^* (z,x;n)$, where
we recall that $c^* (n)$ and $U^* (x \cdot n - c^* (n)t,x ;n)$
denote respectively the unique admissible speed and the unique
pulsating wave in the direction $n$, normalized by
\begin{equation}\label{shift-min}
\min _{x\in \R ^N}U^* (0,x;n) = \frac{1+ \theta}{2}.
\end{equation}

\begin{proof} [Proofs of Theorem \ref{th:continuity_speeds} (ignition case) and Theorem \ref{th:continuity_igwaves}]
We first claim (we postpone the proof to the end of this section)
that
\begin{equation}\label{claim-kappa}
\kappa := \inf_{n \in \mathbb{S}^{N-1}} c^* (n) >0.
\end{equation}

Let us now prove that that $c^*(n)$ is also bounded from above,
using
$$(t,x)\mapsto v(t,x):= \min \{ 1 , \theta + C e^{-\lambda (x \cdot n - 2 a_1 \lambda t)} \}$$ as a
supersolution of \eqref{monostable}. Here $C$ and $\lambda$ are
positive constants to be chosen later, and $a_1$ comes from
hypothesis \eqref{ellipticity}. Indeed, when $v < 1$, it satisfies
\begin{eqnarray}
& & \partial_t v - \mbox{div} \, (A(x) \nabla v ) - q(x) \cdot \nabla v -f(x,v) \vspace{3pt}\nonumber\\
& = &  \left[ 2 a_1 \lambda^2 - (n \cdot A n) \lambda^2 + \lambda \,  \mbox{div} \, (An) + \lambda q \cdot n \right] \times C e^{- \lambda (x \cdot n - 2 a_1 \lambda t)} - f(x,v)  \vspace{3pt}\nonumber\\
& \geq &  \left[ a_1 \lambda^2 -  \lambda \,  |\mbox{div} \, (An)|
-  \lambda |q \cdot n|  - M \right] \times C e^{- \lambda (x \cdot
n - 2 a_1 \lambda t)} > 0,\label{gggg}
\end{eqnarray}
where
\begin{equation}\label{M}
M:= \sup_{x \in \R^N,u \in [0,1]} \frac{ f(x,u)}{|u - \theta|} < +
\infty
\end{equation}
 comes from the Lipschitz continuity of $f$, and the last
inequality holds provided that $\lambda$ is large enough,
independently of $n \in \mathbb{S}^{N-1}$. As $1$ is a solution of
\eqref{monostable}, it is then clear that $v$ is a generalized
supersolution of \eqref{monostable}. Then, choosing $C>0$ so that
$v(t=0,x)$ lies above the traveling wave $u^* (t=0,x;n)=U^*(x\cdot
n,x;n)$ at time 0, we can apply the comparison principle
 and obtain that $c^* (n) \leq 2 a_1 \lambda$. Putting this fact together with \eqref{claim-kappa}, we conclude that
\begin{equation}\label{speed_bounds0}
0< \kappa := \inf_{n \in \mathbb{S}^{N-1}} c^* (n) \leq \sup_{n
\in \mathbb{S}^{N-1}} c^* (n) =: K < +\infty.
\end{equation}

\medskip

We now let some sequence of directions $n_k \to n \in
\mathbb{S}^{N-1}$. As we have just shown, the sequence $c^* (n_k)$
is bounded and, up to extraction of a subsequence, $c^*  (n_k) \to
c > 0$. We also choose the shifts $z_k$ so that, for all $k$,
$\max_{x \in \R^N} U^*  (z_k, x;n_k) =  \theta $. In particular,
recalling that $U^*$ is monotonically decreasing with respect to
its first variable, we have for all $k$ that
$$\forall z\geq z_k,
\forall x\in \R ^N , \quad 0 <U^* (z,x;n_k) \leq \theta.
$$
Then
$$u_{k} (t,x) := U^* (z_k + x \cdot n_k ,x + c^* (n_k) t n_k ;n_k)$$
satisfies
\begin{equation}\label{ignition_continuity1}
\partial_t u_{k}= \mbox{div} \, (A(x) \nabla u_{k}) + q(x) \cdot \nabla u_{k} + c^* (n_k) \nabla u_{k} \cdot
n_k,
\end{equation}
for all $t \in \R$ and all $x$ in the half-space $x \cdot n_k \geq
0$ (recall that $U^*$ solves \eqref{eq-tw} and that, in the
ignition case, $f (x,u)=0$ if $0 \leq u \leq \theta$).

\medskip

Let us now find a supersolution of \eqref{ignition_continuity1} of
the exponential type, namely
\begin{equation}\label{def-sur}\overline{u}_k (t,x):= \phi_k (t,x) \times  e^{-\lambda_0  \, x \cdot n_k},
\end{equation}
where $\phi_k$ will be a well-chosen positive and bounded
function.

For any $n \in \mathbb{S}^{N-1}$, one may define (see
Proposition~5.7 in \cite{Ber-Ham-02}) the principal eigenvalue
problem
\begin{equation*}
\left\{
\begin{array}{l}
-L_{n , \lambda} \phi_{n,\lambda}  =   \mu (n,\lambda) \phi_{n,\lambda} \ \mbox{ in } \R^N , \vspace{3pt}\\
\phi_{n,\lambda}>0 \mbox{ is periodic},
\end{array}
\right.
\end{equation*}
where
$$
L_{n,\lambda} \phi := \mbox{div} \, (A \nabla \phi) + \lambda^2 (n
\cdot A n) \phi - \lambda (\mbox{div} (An \phi) + n\cdot A\nabla
\phi) + q \cdot \nabla \phi - \lambda (q \cdot n + \kappa ) \phi,
$$
with $\kappa >0$ given by \eqref{claim-kappa}. In the sequel, the
eigenfunction $\phi_{n,\lambda}$ is normalized so that
$$
\min_{x \in \mathcal{C}} \phi_{n,\lambda} (x) = \theta .
$$
As stated in Proposition~5.7 of \cite{Ber-Ham-02}, the function
 $\lambda \mapsto \mu (n,\lambda)$ is concave and satisfies, for any~$n$, that $\mu (n, 0)= 0$ (any positive constant is clearly a principal eigenfunction of $-L_{n,0}$), and $\partial_\lambda \mu (n,0) = \kappa >0$.

It follows that one can find some small $\lambda_0 >0$ such that,
for any $n \in \mathbb{S}^{N-1}$,
$$\mu (n,\lambda_0) >0.$$ Indeed, proceed by contradiction and assume that for
any $j \in \mathbb{N}^*$, there exists $n_j$ such that $\mu (n_j,
1/j) \leq 0$. Then, by $\mu(n_j,0)=0$ and by concavity, one has
that $\mu (n_j,\lambda) \leq 0$ for all $\lambda > \frac{1}{j}$.
By uniqueness of the principal normalized eigenfunction, it is
straightforward to check that $\mu (n,\lambda)$ depends continuously on both $n$ and $\lambda$, as well as
$\phi_{n,\lambda}$ with respect to the
uniform topology. Thus, one can pass to the limit and conclude
that $\mu (n_\infty, \lambda) \leq 0$ for some $n_\infty = \lim
n_j$ (up to extraction of a subsequence) and all $\lambda \geq 0$.
This contradicts the fact that $\partial_\lambda \mu
(n_\infty,0)=\kappa >0$.

Notice that, by continuity of the eigenfunction with respect to
$n$ and $\lambda$ in the uniform topology, it is clear that for
any bounded set $\Lambda$,
\begin{equation}\label{eigenfunction_bound}
\max_{n \in \mathbb{S}^{N-1}} \max_{\lambda \in \Lambda} \ \max_{x
\in\mathcal{C}} \phi_{n,\lambda} (x)< +\infty.
\end{equation}

Choosing $\lambda_0$ as above and
$$
\phi_k (t,x) := \phi_{n_k,\lambda_0} (x + c^* (n_k)  t n_k),
$$
in \eqref{def-sur}, one gets that
\begin{eqnarray*}
&& \partial_t \overline{u}_{k}- \mbox{div} \, (A(x) \nabla
\overline{u}_{k}) - q(x) \cdot \nabla \overline{u}_{k} - c^* (n_k)
\nabla \overline{u}_{k} \cdot
n_k \vspace{3pt} \\
& = & \left[ c^* (n_k) n_k \cdot \nabla \phi_{n_k ,\lambda_0}
-L_{n_k,\lambda_0} \phi_k  + \lambda_0 ( - \kappa + c^* (n_k))
\phi_k - c^* (n_k) n_k \cdot \nabla \phi_{n_k ,\lambda_0}\right]
 \times e^{-\lambda_0 (x \cdot n_k)} \vspace{3pt} \\
& = &  \left[ \mu (n_k ,\lambda_0) + \lambda _0 (c^* (n_k) -
\kappa) \right] \overline{u}_k >0.
\end{eqnarray*}
In other words, as announced, $\overline{u}_k$ is a supersolution
of \eqref{ignition_continuity1}.

\medskip
Let us now prove that
\begin{equation}\label{ignition_continuity2}
\forall t\in \R,\;  \forall x \cdot n_k \geq 0 , \quad u_{k} (t,
x) \leq \overline{u}_{k} (t,x) .
\end{equation}
Proceed by contradiction and define a sequence of points
$(t_j,x_j)_{j \in \mathbb{N}}$ such that
$$u_{k} (t_j,x_j) - \overline{u}_{k} (t_j,x_j)\to \sup_{t \in \R, x\cdot n_k \geq 0} ( u_{k} (t,x) - \overline{u}_{k}(t,x) ) >0.$$
Now write $x_j = (x_j \cdot n_k)n_k + y_j$ for any $j \geq 0$.
Note that, since $u_{k} (t,x) $ and $\overline{u}_k (t,x)$ both
tend to $0$ as $x \cdot n_k \to +\infty$ uniformly with respect to
$t$, then $(x_j \cdot n_k)_{j \in \mathbb{N}}$ must be bounded.
Thus, up to extraction of a subsequence, we can assume that $x_j
\cdot n_k \to a \geq 0$ as $j \to \infty$. Moreover, since $y_j$
is orthogonal to $n_k$, since $\phi_{n_k,\lambda _0}$ is periodic
and since $U^*$ is periodic with respect to its second variable,
we can assume without loss of generality that $y_j + c^* (n_k)t_j
n_k  \in \mathcal{C}$ the cell of periodicity. As $y_j$ is
orthogonal to $n_k$ for all $j \in \mathbb{N}$, we can extract a
subsequence such that both $y_j \to y_\infty \in \R^N$ and $t_j
\to t_\infty \in \R$.

Finally, $u_k - \overline{u}_k$ reaches its positive maximum, over
$t \in \R$ and $x \cdot n_k \geq 0$, at $(t=t_\infty, x=an_k +
y_\infty)$. Moreover, as
$$\forall  x \cdot n_k = 0 , \quad u_k (0,x) \leq \theta \leq  \overline{u}_k (0,x) ,$$
the maximum is reached at an interior point, which contradicts the
parabolic maximum principle. Thus, \eqref{ignition_continuity2} is
proved.

\medskip

Now, by standard parabolic estimates and up to extraction of a
subsequence, we can assume that, as $k\to \infty$, the sequence
$u^* \left(t- \frac{z_k}{c^* (n_k)} ,x;n_k \right)=U^*(x\cdot
n_k-c^*(n_k)t+z_k,x;n_k)$ converges locally uniformly, along with
its derivatives, to a solution $u_\infty (t,x)$ of
\eqref{monostable}. Moreover, $u_\infty$ satisfies
$$\forall l \in \Pi_{i=1}^N L_i \mathbb{Z} , \quad u_\infty (t,x) = u_\infty \left( t+ \frac{l \cdot n}{c},x+l\right).$$
In a similar way than the discussion after
Definition~\ref{def:puls} of pulsating waves, this means that
$u_\infty (t,x) = U_\infty (x\cdot n- ct,x)$ where $U_\infty(z,x)$
is periodic with respect to its second variable and satisfies
\begin{equation*}
\begin{array}{l}\mbox{div}_x (A \nabla_x U) + (n \cdot An)\, \partial_{zz} U + \mbox{div}_x ( An \partial_z U) + \partial_z (n\cdot A \nabla_x U) \vspace{3pt}\\
\qquad \qquad + \, q \cdot \nabla_x U +  (q \cdot n) \, \partial_z
U   + c\partial _z U+f(x,U)=0\quad \text{ on } \R\times \R ^N.
\end{array}
\end{equation*}
It is then straightforward to retrieve that the sequence $U^*
(z+z_k,x;n_k)$ also converges, along with its derivatives, to this
function $U_\infty(z,x)$. In particular, $U_\infty$ is
nonincreasing with respect to its first variable, and satisfies
the inequalities
$$0 \leq U_\infty (z,x) \leq 1 .$$
Furthermore, noticing that $u^* \left(t- \frac{z_k}{c^* (n_k)}
,x;n_k \right)= u_{k} (t,x- c^* (n_k) tn_k )$,
 it follows from passing to the limit in \eqref{ignition_continuity2}, and thanks to \eqref{eigenfunction_bound},  that
$$u_\infty (t,x) \leq A e^{-\lambda_0 (x\cdot n - c t)},$$
for some $A >0$ and all $x \cdot n \geq c t$.

Thus, $U_\infty (x \cdot n - ct,x) \leq A e^{-\lambda_0 ( x \cdot
n - c
 t)}$, for all $t\in \R$ and $x \cdot n \geq ct$. This means that $U_\infty (z,x)$ converges exponentially to 0 as $z \to +\infty$, uniformly with respect to its second variable:
\begin{equation}\label{ignition_continuity3}
\forall z\geq 0,\; \forall x\in \R ^N, \quad U_\infty (z,x) \leq A
e^{-\lambda_0 z}.
\end{equation}
By monotonicity with respect to its first variable,
$U_\infty(z,x)$ converges as $z \to -\infty$ to some periodic
function $p (x)$. Or, equivalently, $u_\infty (t,x)$ converges as
$t \to +\infty$ to the same function $p (x)$. By standard
parabolic estimates, we get that $p (x)$ is a periodic and
stationary solution of \eqref{monostable}. Let us show that $p
\equiv 1$.  From our choice of the shifts $z_k$ and up to
extraction of another subsequence, there exists some $x_\infty$
such that $U_\infty (0,x_\infty) =\theta$, hence $\max p \geq
\theta $. Assume first that $\max p = \theta$. Then $u_\infty
(t,x) \leq \theta$ for all $t \in \R$ and
 $x \in\R^N$ and, by the strong maximum principle, $u_\infty \equiv \theta$.
 This contradicts the inequality~\eqref{ignition_continuity3} above.
 Therefore, $\max p > \theta$. Using again the strong maximum principle and
 the fact that $f (x,u)=0$ for all $u \leq \theta$ and $x \in \R^N$, we reach
 another contradiction if $\min p \leq \theta$. Therefore $\min p > \theta$ and, thanks to part $(ii)$ of our ignition Assumption~\ref{hyp:ignition}, $p \equiv 1 $
 the unique periodic stationary solution of
 \eqref{monostable} above $\theta$.

From the above analysis, we conclude that $U_\infty (\cdot ,
\cdot) = U^* (\cdot + Z, \cdot ;n)$ the unique pulsating traveling
wave in the $n$ direction with speed $c=c^* (n)$, where $Z$ is the
unique shift such that $ \max_{x \in \R^N} U^* (Z, x;n) = \theta$.
This in fact proves, by uniqueness of the limit, that the whole
sequence $c^*(n_k)$ converges to $c^*(n)$, and that the whole
sequence $U^* (\cdot +z_k ,\cdot;n_k)$ converges locally uniformly
to $U^*(\cdot + Z,\cdot;n)$. This in particular shows the
continuity of the map $n\mapsto c^*(n)$, that is
Theorem~\ref{th:continuity_speeds} in the ignition case.

\medskip

Let us now conclude the proof of
Theorem~\ref{th:continuity_igwaves}. Let us first prove that the
 sequence of shifts $z_k$ is bounded.  The normalization
 \eqref{shift-min} implies that $U^*(0,y_k;n_k)=\frac{1+\theta}2$,
 for some $y_k \in \mathcal C$ that (up to some subsequence)
 converges to some $y\in \mathcal C$. Since $U ^*(z_k,y_k;n_k)\to U^*(Z,y;n)\leq \theta$ and $U^*
(0,y;n) = \frac{1+\theta}{2}$, the monotonicity of traveling waves
enforces $z_k \geq 0$ for $k$ large enough. Now proceed by
contradiction and assume that (up to some subsequence) $z_k \to
+\infty$. Then, for all $-z_k \leq z \leq 0$,
$$
 U^* (z + z_k,y_k;n_k) \leq U^*(0,y_k;n_k)= \frac{1+
\theta}{2}.
$$
Passing to the limit as $k \to +\infty$, we get that
$$
U_\infty (z,y) \leq \frac{1 +\theta }{2}, $$ for all $z \leq 0$.
This contradicts the fact that $U_\infty$ is a pulsating traveling
wave and converges to $1 $ as $z \to -\infty$. 

From the boundedness of the sequence $z_k$, we can now rewrite the
convergence as follows: the sequence $U^* (\cdot ,\cdot;n_k)$
converges locally uniformly to $U^* (\cdot,\cdot;n)$. It now
remains to prove that this
  convergence is in fact uniform with respect to both variables. Note first that uniformity with respect to the second
variable immediately follows from the periodicity. Furthermore,
for a given $\delta >0$, let $K >0$ be such that, for any $x \in
\R^N$,
\begin{equation}\label{loin}
0 \leq U^* (z, x; n ) \leq \frac{\delta}{2} \ \mbox{ and } \ 1 -
\frac \delta 2 \leq U^* ( -z, x; n ) \leq 1, \textrm{ for all } z
\geq K.
\end{equation}
From the locally uniform convergence with respect to the first
variable, we have, for any $k$ large enough,
$$\| U^* (\cdot,\cdot;n_k) - U^* ( \cdot, \cdot; n ) \|_{L^\infty([-K,K]\times \R)} \leq \frac \delta 2 .$$
In particular, $U^* ( K, x; n _k) \leq \delta$ and $1 -\delta \leq
\ U^* ( -K, x; n_k )$, so that, by monotonicity with respect to
the first variable, for any $x \in \R^N$ and $k$ large enough,
\begin{equation}\label{loin2}
0 \leq U^* (z, x; n_k ) \leq \delta\ \mbox{ and } \ 1-   \delta
\leq U^*  ( -z, x; n_k ) \leq 1 , \textrm{ for all } z \geq K.
\end{equation}
 Combining \eqref{loin} and \eqref{loin2}, we get
$$
\| U^* (\cdot,\cdot;n_k) - U^* ( \cdot, \cdot; n )
\|_{L^\infty((-\infty,-K)\cup (K, \infty)\times \R)} \leq  \delta,
$$
for any $k$ large enough. As a result the convergence of $U^*
(\cdot,\cdot;n_k)$ to $U^* (\cdot,\cdot;n)$ is uniform in $\R
\times \R^N$. This ends the proof of the continuity of ignition
waves, that is Theorem~\ref{th:continuity_igwaves}.
\end{proof}

\begin{proof}[Proof of claim \eqref{claim-kappa}]
Proceed by contradiction and assume that there exists a sequence
$n_k \in \mathbb{S}^{N-1}$ such that $c^* (n_k) \to 0$.

Now for any $k$, 
recall that the pulsating wave is normalized by
\begin{equation}\label{shift-min2}
\min _{x\in \R^N} U^* (0,x;n_k) = \frac{1+\theta}{2}.
\end{equation} Up to extraction of a subsequence, we can assume as
above that $n_k \to n$ and
$$u^* \left(t,x;n_k \right) \to u_\infty (t,x),$$
where the convergence is understood to hold locally uniformly, and
$u_\infty (t,x)$ is a solution of~\eqref{ignition_continuity1}. By
the strong maximum principle, it is clear that $0 < u_\infty < 1$.
We also know, by the monotonicity of $U^* (\cdot ,\cdot;n_k)$ with
respect to its first variable, by \eqref{shift-min2} and by
passing to the limit, that
$$
u_\infty (t,x) \geq \frac{1+\theta}{2}, \quad \forall x\cdot n
\leq 0.
$$
 Applying Weinberger's result (see
Theorem \ref{Wein02_cauchy} as well as Remark~\ref{rem_bug}), we
get that the solution spreads at least at speed $c^* (n)$. In
particular, as $t\to+\infty$, $u_\infty(t,x)$ converges locally
uniformly to~1.

On the other hand, we fix $x \in \R^N$ and $s \geq 0$, then we let
some vector $l \in \Pi_{i=1}^N L_i \mathbb{Z}$ be such that $l
\cdot n > 0$. In particular, for any large $k$, one also has that
$l \cdot n_k \geq \frac{l \cdot n}{2} >0.$ Then, for all large
$k$, using the fact that $c^* (n_k) \to 0$ and the monotonicity of
$u^* (\cdot , \cdot ; n_k)$ with respect to its first variable, we
have that
$$u^* \left( s ,x ;n_k \right) \leq u^* \left( \frac{l \cdot n_k}{c^* (n_k)}  , x  ;n_k \right) = u^* \left(0 ,x - l ;n_k \right).$$
By passing to the limit as $k \to +\infty$, we obtain that
$$
u_\infty (s,x) \leq u_\infty (0,x-l)<1,
$$
for all $x \in \R^N$ and $s \geq 0$. This contradicts the locally
uniform convergence of $u_\infty(t,x)$ to $1$ as $t \to + \infty$.
The claim is proved.\end{proof}

\section{Continuity of the monostable minimal speed}\label{s:monostable-speed}

Let us here consider a periodic nonlinearity $f$ of the monostable
type, namely satisfying Assumption~\ref{hyp:monostable}. We will
prove the continuity of the mapping $n  \mapsto c^* (n)$, that is
Theorem~\ref{th:continuity_speeds}. To do so, we introduce a
family $f_\varepsilon (x,u)$, for small $\varepsilon
>0$, of ignition nonlinearities which serve as approximations from
below of the  monostable  nonlinearity $f(x,u)$. Our aim is to
prove that, by passing to the limit as $\varepsilon \to 0$, we
indeed retrieve the dynamics of the monostable equation. This will
be enough to prove Theorem~\ref{th:continuity_speeds}.

\medskip

The family $(f_\varepsilon )_\varepsilon$, for small enough
$\varepsilon >0$, is chosen as follows:
\begin{equation*}
\forall x \in \R^N, \  \left\{\begin{array}{rl}
\forall u \in [-\varepsilon, 0], & f_\varepsilon  (x,u) = 0   \vspace{3pt}\\
 \forall  u \in \left[0 ,  1 -  \varepsilon \right], & f_\varepsilon  (x,u) = f (x,u) \vspace{3pt}\\
  \forall u \in \left[1 - \varepsilon, 1 - \frac{\varepsilon}{2} \right], & f_\varepsilon  (x,u) =  f \left(x,1 - \varepsilon + 2 (u-(1-\varepsilon))
  \right).
\end{array}\right.
\end{equation*}
Notice that $\Vert f_\ep -f\Vert _{L^\infty(-\ep,1)}\to 0$ as $\ep
\to 0$, and that, thanks to Assumption \ref{hyp:monostable}
$(iii)$,
 $f_\ep$ lies below $f$ and  $0<\ep<\ep'$ implies $f_\ep \geq
 f _{\ep '}$. Also, the equation
\begin{equation}\label{ignition_approx}
\partial _t u=  \mbox{div} (A (x) \nabla u)  + q (x) \cdot \nabla u + f_\varepsilon ( x,u),
\end{equation}
where $u$ is to take values between $-\varepsilon$ and $1-\frac
\varepsilon 2$, is of the ignition type in the sense of
Assumption~\ref{hyp:ignition} (where 0, $\theta$, 1 are replaced
by $-\varepsilon$, 0 and $1 - \frac{\varepsilon}{2}$
respectively). In particular, for each $n \in \mathbb{S}^{N-1}$,
there exists a unique ignition pulsating traveling wave
$$u_\varepsilon^* (t,x;n) = U_\varepsilon^* (x \cdot n -
c^*_\varepsilon (n) t,x;n)$$ of \eqref{ignition_approx} in the $n$
direction with speed $c^*_\varepsilon (n) >0$, normalized by
$$
\min _{x\in \R ^N}U_\ep ^*(0,x;n)=\frac 12.
$$
 Furthermore, we
have already proved in the previous section that the mappings $n
\mapsto  c^*_\varepsilon (n)$ and
 $n\mapsto U_\varepsilon^* (\cdot, \cdot;n)$ are continuous (with respect to the uniform topology).

\begin{theo}[Convergence of speeds] \label{th:uc_speeds} Assume that $f$ is of the spatially periodic monostable type, i.e.
$f$ satisfies \eqref{periodicity} and
Assumption~\ref{hyp:monostable}. Let  $f_\varepsilon(x,u)$ be
defined as above.

Then, as $\varepsilon \to 0$,
 $c^*_\varepsilon (n) \nearrow c^* (n)$ uniformly with respect to $n \in
\mathbb{S}^{N-1}$.
\end{theo}

As mentioned before, pointwise convergence was shown
in~\cite{Ber-Ham-02}, where the goal was to prove existence of
monostable traveling waves for
 the range of speeds $[c^* (n) , +\infty)$. Here we prove that the convergence is actually uniform,
 which together with the continuity of speeds in the ignition case, immediately insures the continuity of $n \mapsto c^* (n)$, that is Theorem~\ref{th:continuity_speeds} in the monostable case.

\begin{proof} First note that, for any fixed $n\in \mathbb{S}^{N-1}$ and $\varepsilon >0$, $c^*_{\varepsilon} (n) \leq c^* (n)$. Indeed, recalling that $U_\varepsilon^* (z,x;n)$
connects $1 - \frac{\varepsilon}{2}$ to $-\varepsilon$, one can
find some shift $Z\in \R$ such that $U^*_\varepsilon (z+Z,x;n)
\leq U^* (z,x;n)$, where $U^*$ denotes a monostable pulsating
traveling wave --- connecting 1 to 0--- with the minimal speed
$c^* (n)$. By a comparison argument, it follows that
$c^*_\varepsilon (n) \leq c^* (n)$. It is also very similar to
check that, for any $n \in \mathbb{S}^{N-1}$,
$0<\varepsilon<\varepsilon '$ implies $c_\varepsilon ^*(n)\geq
c_{\varepsilon '}^*(n)$.

Let us now consider some sequences $\varepsilon_k \to 0$ and $n_k
\to n$. Consider the estimate \eqref{speed_bounds0} where $\kappa$ and $K$ should {\it a priori} depend on $\varepsilon$. First, it is clear from the above that $\kappa (\varepsilon) := \inf_n  c^*_\varepsilon (n)$ is
   nonincreasing with respect to $\varepsilon$. Also, since
   $$
   \sup _{0<\ep\leq \ep _0} M_\ep:=\sup  _{0<\ep\leq \ep _0}\quad  \sup_{x \in \R^N,u \in [-\ep,1-\frac \ep 2]} \frac{ f_\ep (x,u)}{|u|} < +
\infty
$$
(compare with \eqref{M}), arguing as we did to derive
\eqref{speed_bounds0}, we see that $K (\varepsilon) := \sup_n
c^*_\varepsilon (n)$ is uniformly
    bounded from above. As a result, we have
    \begin{equation}\label{speed_bounds0bis}
0< \kappa := \inf_{0<\ep\leq \ep _0} \inf_{n \in \mathbb{S}^{N-1}} c^* (n) \leq\sup _{0<\ep \leq \ep _0} \sup_{n
\in \mathbb{S}^{N-1}} c^* (n) =: K < +\infty.
\end{equation}
Hence, we can assume,
 up to extraction of a subsequence, that $c^*_{\varepsilon_k} (n_k)\to c_\infty >0 $ as $k \to \infty$. 
    In order to prove Theorem~\ref{th:uc_speeds}, we have to prove that $c_\infty = c^* (n)$.

    \medskip

We begin by showing that $ U^*_{\varepsilon_k} (z,x;n_k)$
converges as $k \to \infty$ to a monostable pulsating traveling
wave of \eqref{monostable},
  up to extraction of a subsequence. Indeed, proceeding as before, one can use standard parabolic estimates to
extract a converging subsequence of pulsating ignition traveling
waves, such that
$$
U^*_{\varepsilon_k} (z,x;n_k) \to U_\infty (z,x),
$$
as $k \to +\infty$ locally uniformly with respect to $(z,x) \in \R
\times \R^N$. Furthermore, $0 \leq U_\infty (z,x ) \leq 1$ solves
\eqref{eq-tw} with $c = c_\infty$, is nonincreasing with respect
to $z$, periodic with respect to $x$, and satisfies $\min _{x\in
\R ^N}U_\infty (0, x) = \frac{1}{2}$. In particular, $U_\infty$
converges as $z \to \pm \infty$ to two periodic stationary
 solutions of \eqref{monostable}, which under the monostable Assumption~\ref{hyp:monostable} can only be 0 and 1. We can conclude that $U_\infty$
 is a monostable pulsating traveling wave with speed $c_\infty$, hence $c_\infty \geq c^* (n)$.

We now prove that $c_\infty = c^* (n)$. Notice that
$f_\varepsilon$ lies below $f$ but, since the direction varies, we
cannot use a simple comparison argument to conclude that $c_\infty
\leq c^* (n)$. Instead, we will use a sliding method as
in~\cite{Ber-Ham-02}. To do so, we shall need the following lemma.
\begin{lem}[Some uniform estimates]
There exists $C>0$ such that, for any small $\varepsilon >0$ and
$n \in \mathbb{S}^{N-1}$, the ignition pulsating traveling wave
$U_\varepsilon^* (z,x;n)$ satisfies
$$| \partial_{zz} U^*_\varepsilon (\cdot, \cdot;n) |  \leq - C  \partial_{z} U^*_\varepsilon (\cdot ,\cdot ,n),
\quad| \nabla_x \partial_z U^*_\varepsilon(\cdot ,\cdot ;n) | \leq
- C
\partial_z U^*_\varepsilon (\cdot,\cdot,n).$$
\end{lem}
\begin{proof}
Let us define $u^*_\varepsilon (t,x):= U^*_\varepsilon (x \cdot n
- c^*_\varepsilon (n) t,x;n)$. Then $v(t,x): =
\partial_t u^* _\varepsilon(t,x)>0$ satisfies
$$
\partial_t v = \mbox{div} (A (x) \nabla v) + q (x) \cdot \nabla v + v\, \partial_u f_\varepsilon
(x,u^*_\varepsilon), \quad \text{ a.e. in $\R \times \R ^N$.}
$$
From our definition of the ignition approximation
$f_\varepsilon(x,u)$, we see that  $\Vert \partial_u f_\varepsilon
\Vert _{L^\infty( \R^N \times (-\ep,1-\frac \ep 2))}$ is uniformly
bounded, independently on small $\varepsilon >0$ and $n\in \mathbb
S ^{N-1}$. Therefore, from the interior parabolic $L^p$-estimates
(see \cite[Theorem 48.1]{Qui-Sou} for instance) and Sobolev
embedding theorem,  one gets
\begin{equation}\label{schauder}
\forall (t_0,x_0) \in \R \times \R^N \, ,  \ | \partial_t v
(t_0,x_0) | + | \nabla_x v (t_0,x_0) | \leq C_1 \max_{t_0 -1 \leq
t \leq t_0 , |x - x_0| \leq 1 } v (t,x),
\end{equation}
for some $C_1 >0$ which is independent on $t_0$, $x_0$, small
$\varepsilon>0$ and $n\in \mathbb S ^{N-1}$.

Furthermore, for any $n \in \mathbb{S}^{N-1}$, choose $N$ integers
$k_i (n) \in \{ -1,0,1 \}$ such that
$$k (n)L \cdot n = \max_{k_1 , ..., k_N \in \{-1,0,1\} } (k_1  L_1 , ... k_N  L_N) \cdot n ,$$
where $k (n)L :=(k_1 (n) L_1 , ... k_N (n) L_N)$. Then
$$
0 < \inf_{n \in \mathbb{S}^{N-1}} k(n)L \cdot n \leq \sup_{n \in
\mathbb{S}^{N-1}} k(n)L \cdot n < +\infty,
$$ and hence, thanks to
\eqref{speed_bounds0bis},
$$
0 < \inf_{0<\varepsilon \leq \varepsilon_0, n \in \mathbb{S}^{N-1}} \frac{k(n) L \cdot n}{c^*_\varepsilon (n)} \leq
 \sup_{0<\varepsilon \leq \varepsilon_0 ,  n \in \mathbb{S}^{N-1}} \frac{k(n) L \cdot n}{c^*_\varepsilon (n)} < + \infty.
 $$
 By the parabolic Harnack inequality for strong solutions (see \cite[Chapter VII]{Lie}
 for instance), we get
\begin{equation}\label{harnack}\forall (t_0,x_0) \in \R \times \R^N \, , \
\max_{t_0 -1 \leq t \leq t_0 , |x - x_0| \leq 1 } v (t,x) \leq C_2
v \left( t_0 + \frac{ k(n) L \cdot n}{c^*_\varepsilon (n)} , x_0 +
k(n) L \right), \end{equation} for some $C_2
>0$ which is also independent on $t_0$, $x_0$, small
$\varepsilon>0$ and $n\in \mathbb S ^{N-1}$.

 Combining \eqref{schauder}, \eqref{harnack} and the space-time
periodicity of the traveling wave, we get
$$
\forall (t_0,x_0) \in \R \times \R^N \, , \ |\partial_t v (t_0,x_0
) | + | \nabla_x v (t_0,x_0) |  \leq C_3v(t_0,x_0),
$$
with $C_3=C_1C_2$. Now recall that $U^*_\varepsilon (z,x;n) =
u^*_\varepsilon \left( \frac{x\cdot n-z }{c^*_\varepsilon (n)} , x
\right).$ Thus
$$| \partial_{zz} U^*_\varepsilon | = \frac{1}{c^*_\varepsilon (n)^2} |\partial_t v | \leq \frac{C_3}{c^*_\varepsilon (n)^2} v =-\frac{C_3}{c_\ep^*(n)} \partial_{z} U^*_\varepsilon,$$
$$| \nabla_x \partial_z U^*_\varepsilon | \leq \left| \frac{-1}{c^*_\varepsilon (n)^2} \, \partial_t v \, n - \frac{1}{c^*_\varepsilon (n)} \nabla_x v \right| \leq
\left(\frac 1{c^*_\ep(n)^2}+\frac 1{c^*_\ep(n)}\right)C_3v=-\left(\frac 1
{c^*_\ep(n)}+1\right)C_3 \partial_z U^*_\varepsilon.
$$
Since $\kappa=\inf_{0 < \varepsilon \leq \varepsilon_0} \inf_{n\in
\mathbb{S}^{N-1}} c^*_\varepsilon (n)
>0$, this proves the lemma.
\end{proof}

Let us now go back to the proof of Theorem~\ref{th:uc_speeds}.
Proceed by contradiction and assume that
 $c_\infty \geq c^* (n) + \delta$ for some $\delta >0$. We
 plug $U^*_{\varepsilon_k}(\cdot,\cdot;n_k)$ into equation~\eqref{eq-tw} satisfied by
 $U^*(\cdot,\cdot;n)$ and, thanks to the above lemma, get
\begin{eqnarray}
& &  \mbox{div}_x (A \nabla_x U^*_{\varepsilon_k}) + (n \cdot
An)\, \partial_{zz} U^*_{\varepsilon_k} +
 \mbox{div}_x ( An \,  \partial_z U^*_{\varepsilon_k})  + \partial_z (n\cdot A \nabla_x U^*_{\varepsilon_k}) \vspace{3pt} \nonumber \\
&& \qquad \quad + \, q \cdot \nabla_x U^*_{\varepsilon_k} +  (q \cdot n) \, \partial_z U^*_{\varepsilon_k}   + c^*(n)\partial _z U^*_{\varepsilon_k} +f(x,U^*_{\varepsilon_k}) \vspace{3pt}\nonumber \\
& = &  (n \cdot An - n_k \cdot A n_k) \partial_{zz} U^*_{\varepsilon_k} + \mbox{div}_x ((An - An_k) \times \partial_z U^*_{\varepsilon_k}) \nonumber \\
& & \qquad \quad + \partial_z ((n - n_k)\cdot A \nabla_x U^*_{\varepsilon_k}) + (q \cdot (n-n_k))\partial_z U^*_{\varepsilon_k} \nonumber  \\
& & \qquad \quad + (c^* (n) -c^*_{\varepsilon_k} (n_k)) \partial_z U^*_{\varepsilon_k}  + f (x,U^*_{\varepsilon_k}) - f_{\varepsilon_k} (x,U^*_{\varepsilon_k}) \vspace{3pt} \nonumber \\
& \geq & \left[ 4 a_2 C | n- n_k|  + \mbox{div}_x \, A (n-n_k)  +
|q| | n- n_k |  - \frac \delta 2\right] \partial_{z}
U^*_{\varepsilon_k}
\vspace{3pt}\nonumber \\
 & \geq & -\frac{ \delta}{3}
\partial_z U^*_{\varepsilon_k} >0,\label{ineg}
\end{eqnarray}
provided $k$ is large enough, and where $a_2>0$ comes from
\eqref{ellipticity}. We now use the sliding method. From the
asymptotics
$$
U^*_{\varepsilon_k} (+\infty, \cdot ;n_k) = - \varepsilon _k < 0 =
U^* (+\infty, \cdot ;n),
$$
$$
U^*_{\varepsilon_k} (-\infty,\cdot ;n_k) = 1 - \frac{ \varepsilon
_k}{2} < 1 = U^* (+\infty, \cdot;n),
$$
one can define
$$\tau_0 := \inf \{ \tau:\,  U^*_{\varepsilon_k} (z + \tau,x ;n_k) < U^* (z,x;n), \forall z \in \R, \forall x \in \R ^N \} \in \R.$$
Then, using again the asymptotics as well as the periodicity with
respect to $x$ of any pulsating wave, there is some first touching
point $(z_0,x_0)\in\R \times \R^N$ such that
$$
U^*_{\varepsilon_k} (z_0 + \tau _0,x_0;n_k) = U^* (z_0,x_0;n),
\quad \text{ and }\quad  U^*_{\varepsilon_k} (\cdot + \tau
_0,\cdot;n_k) \leq U^* (\cdot,\cdot;n).$$ Substracting the
equation~\eqref{eq-tw} satisfied by $U^*(z,x;n)$ to the
inequality~\eqref{ineg} satisfied by $U^*_{\varepsilon_k}(z+\tau
_0,x;n_k)$ above, and estimating it at point $(z_0,x_0)$, we get
that
$$
0 \geq - \frac{\delta}{3} \partial_z U^*_{\varepsilon_k}(z_0+\tau
_0,x_0;n_k)
>0,
$$
a contradiction. Hence, $c_\infty = c^* (n)$, and the convergence
of $c^*_\varepsilon(n)$ to $c^*(n)$ is uniform.
\end{proof}

\begin{rem}[On the convergence of profiles]
The argument above also shows that the ignition traveling waves
converge locally uniformly,
 up to a subsequence, to a traveling wave with minimal speed of the monostable equation.
  Proceeding as in Section~\ref{s:ignition} and thanks to
  the monotonicity of traveling waves, one can check that this convergence is actually uniform in time
   and space. In particular, they do not flatten as the parameter $\varepsilon \to 0$. However, as the uniqueness
    of the monostable traveling wave with minimal speed is not known \cite{Ham-Roq}, we cannot conclude on the convergence of the whole sequence.
\end{rem}

\section{The uniform lower spreading}\label{s:lower-spreading}

In this section and the next, we will prove
Theorem~\ref{th:unif_spreading} under the monostable assumption.
The easier ignition case will be dealt with in the last section.

We begin here with the uniform lower spreading
property~\eqref{conclusion1} of Theorem~\ref{th:unif_spreading}.
The argument again relies on the approximation from below by an
ignition type problem, and follow the footsteps of the proof of
Theorem~\ref{th:uc_speeds}.

\begin{proof}[Proof of \eqref{conclusion1}]
Recall that $f_\varepsilon(x,u)$ is an ignition type nonlinearity
which approximates $f(x,u)$ from below as $\varepsilon \to 0$. We
still
 denote $u_\varepsilon^* (t,x) = U^*_\varepsilon (x\cdot n - c^*_\varepsilon (n) t,x;n)$ the unique ignition pulsating traveling wave
  of \eqref{ignition_approx} in the direction $n$, normalized by $\min _{x\in \R ^N}U_\ep ^*(0,x;n)=\frac 12$.

As $f_\varepsilon \leq f$, it is clear that $u^*_\varepsilon$ is a
subsolution of \eqref{monostable}, whose speed is arbitrary close
to $c^* (n)$ as $\varepsilon \to 0$ thanks to
Theorem~\ref{th:uc_speeds}. This leads back to Weinberger's result
\cite{Wein02}, namely the fact that for any planar-like initial
datum  in the $n$ direction, the solution of \eqref{monostable}
spreads with speed \lq\lq at least'' $c^* (n)$ in the $n$
direction.

Let us now make this spreading property uniform with respect to
the family of solutions $(u_n)_{n \in \mathbb{S}^{N-1}}$, as
stated in Theorem~\ref{th:unif_spreading}. In the following $\mu$
and $K$ are as in assumption~\eqref{hyp2} (monostable case). Let
$\alpha
>0$ and $\delta >0$ be given. In view of assumption \eqref{hyp0} and the comparison principle we
have $u_n(t,x)\leq 1$. Hence to prove \eqref{conclusion1}, we need
to find $\tau
>0$ so that
\begin{equation}\label{conclusion1-dessous} \inf_{n \in
\mathbb{S}^{N-1}} \ \inf_{x \cdot n \leq (c^*(n) - \alpha)t} u_n
(t,x) \geq 1 - \delta,
\end{equation}
holds for all $t\geq \tau$.

\medskip

In view of Theorem~\ref{th:uc_speeds}, we can fix $\varepsilon>0 $
small enough so that, for all $n\in \mathbb S ^{N-1}$,
\begin{equation}\label{speed-prime}
c^*_{\varepsilon}(n) \geq c^* (n) - \frac \alpha 2 .
\end{equation}
We then claim that one can find some $t_\varepsilon>0$ such that
\begin{equation}\label{claim:proof_uniformspread1}
u_n (t_\varepsilon,x) \geq 1 - \frac{\varepsilon}{2},
\end{equation}
for all $n \in \mathbb{S}^{N-1}$ and all $x$ such that $x \cdot n
\leq -K$. We insist on the fact that $t_\varepsilon$ does not
depend on $n \in \mathbb{S}^{N-1}$. To prove
\eqref{claim:proof_uniformspread1}, let us define
$$
\mathcal{S}= \{x \in \R^N : \; x \cdot n \leq c^* (n) \mbox{ for all } n \in \mathbb{S}^{N-1} \}.
$$
We know from Theorem \ref{th:continuity_speeds} that the mapping
$n \mapsto c^* (n)$ is positive and continuous, hence
$\mathcal{S}$ has nonempty interior. It is then known (see
Theorem~2.3 in \cite{Wein02}, as well as Remark~\ref{rem_bug}
above) that for compactly supported initial data \lq\lq with large
enough support'', the associated solution of \eqref{monostable}
converges locally uniformly to 1 as $t\to +\infty$ (in fact, even
uniformly on the expanding sets $t \mathcal{S}'$ for any subset
$\mathcal{S'}$ of the interior of $\mathcal{S}$; also, under the
additional assumption that $0$ is linearly unstable with respect
to the periodic problem, this is even  true for any non trivial
and compactly supported initial datum, regardless of its
size~\cite{Ber-Ham-Roq1}, \cite{Ber-Ham-Nad08}). More precisely,
let $u_R$ be the solution of \eqref{monostable} associated with
the initial datum $u_{0,R} (x) = \mu \times \chi_{B_R} (x)$, where
$R$ is a large but fixed positive constant (depending on $\mu$)
which we can assume to be larger than $2 \sqrt{N} \max_i L_i$.
Here $B_R$ denotes the ball of radius $R$ centered at the origin.
Then $u_R$ converges locally uniformly to 1 as $t \to +\infty$. In
particular, $$u_R (t_\varepsilon,x) \geq 1 -
\frac{\varepsilon}{2},$$ for some $t_\varepsilon >0$ and all $x
\in B_{2R}$.  Besides, for $x_0 \in \Pi_{i =1}^{N} L_i \mathbb{Z}$
such that $x_0 \cdot n \leq -K-R$, we have --- thanks to
\eqref{hyp2}--- that $u_n(0,x+x_0)\geq u_R(0,x)$. Then, by the
comparison principle,
$$
\forall x \in B_{2R}, \quad u_n (t_\varepsilon,x+x_0)  \geq
u_R(t_\varepsilon,x)\geq 1-\frac \varepsilon 2 .
$$
Since $R > 2 \sqrt{N} \max_i L_i$, for all $x \cdot n \leq -K$,
there exists $x_0  \in \Pi_{i =1}^{N} L_i \mathbb{Z}$ such that
$x_0 \cdot n \leq -K -R$ and $x \in B_{2R} (x_0)$. Thus, we obtain
$u_n (t_\varepsilon ,x) \geq 1 - \frac{\varepsilon}{2}$, for all
$n \in \mathbb{S}^{N-1}$ and $x \cdot n \leq -K$, that is
claim~\eqref{claim:proof_uniformspread1}.

\medskip

Now, recall that $U^*_{\varepsilon} (\cdot , \cdot ;n)$ is the
pulsating traveling wave of equation~\eqref{ignition_approx} in
the direction~$n$, connecting $1-\frac{\varepsilon}{2}$ to
$-\varepsilon$. Hence, it follows from
\eqref{claim:proof_uniformspread1} that, for any $n\in \mathbb S
^{N-1}$, one can find some shift $Z_n$ such that
\begin{equation}\label{ineq_1234}
u_n (t_\varepsilon,x) \geq  U^*_{\varepsilon} ( x \cdot n -
c^*_{\varepsilon} (n) t_\varepsilon + Z_n,x;n).
\end{equation}
Actually, it suffices to select
$$
Z_n := \min \{ z \in \R:\; \  \min_{x \in \mathcal{C}} U^*_{\varepsilon} ( - K  -
c^*_{\varepsilon} (n) t_\varepsilon + z,x;n)  \leq 0 \} \in
(0,\infty).
$$ Moreover, from the
uniform continuity of ignition traveling waves with respect to the
direction, namely Theorem~\ref{th:continuity_igwaves}, it is
straightforward that the family $(U^*_{\varepsilon} (z,x;n))_{n
\in \mathbb{S}^{N-1}}$ converges to $-\varepsilon$ as $z \to
+\infty$ uniformly with respect to $n \in \mathbb{S}^{N-1}$.
Therefore, we can also define the bounded real number $ Z :=
\sup_{n \in \mathbb{S}^{N-1}} Z_n\in (0,\infty)$, so that
\eqref{ineq_1234} is improved to
$$
\forall n\in \mathbb S ^{N-1}, \quad u_n (t_\varepsilon,x) \geq
U^*_{\varepsilon} ( x \cdot n - c^*_{\varepsilon} (n)
t_\varepsilon + Z,x;n).
$$
Then we can apply the parabolic comparison principle to get
\begin{equation}\label{ineq_12345}
\forall t \geq t_\varepsilon, \forall x \in \R^N,  \forall n \in
\mathbb{S}^{N-1}, \quad u_n (t,x) \geq U^*_{\varepsilon} (x \cdot
n - c^*_{\varepsilon} (n) t + Z , x;n).
\end{equation}
Therefore it follows from \eqref{speed-prime}, \eqref{ineq_12345}
and the monotonicity of the front that
\begin{equation}\label{low_uspread1}
u_n (t,x) \geq U^*_\varepsilon \left( -\frac{\alpha}{2} t + Z ,
x;n \right),
\end{equation}
for all $n \in \mathbb{S}^{N-1}$, all $t \geq t_\varepsilon$ and
all $x$ such that $x \cdot n \leq (c^* (n) - \alpha)t$. Using
again the uniform continuity of ignition traveling waves with
respect to the direction, namely
Theorem~\ref{th:continuity_igwaves}, one can find some shift
$Z'>0$ such that, for all $n \in \mathbb{S}^{N-1}$,
\begin{equation}\label{low_uspread2}
z\leq - Z' \ \Rightarrow \ U^*_\varepsilon (z,x;n) \geq 1 -
\varepsilon.
\end{equation}
Up to decreasing $\varepsilon$, we can assume that
$\varepsilon<\delta$ without loss of generality. Now choose $\tau
\geq t_\varepsilon$ such that $-\frac{\alpha}{2} \tau + Z \leq
Z'$. Then, we get from \eqref{low_uspread1} and
\eqref{low_uspread2} that
$$
u_n (t,x) \geq 1 - \delta,
$$
for all $n \in \mathbb{S}^{N-1}$,  $t \geq \tau$ and $x$ such that
$x \cdot n \leq (c^* (n) - \alpha) t$. We have thus proved
\eqref{conclusion1-dessous}, and hence \eqref{conclusion1}.
\end{proof}

\section{The uniform upper spreading}\label{s:upper-spreading}

We conclude here the proof of Theorem \ref{th:unif_spreading} (monostable case), by
proving the uniform upper spreading~\eqref{conclusion2} in
subsection \ref{ss:with}. Then in subsection \ref{ss:without} we
again prove \eqref{conclusion2} --- together with the uniform
lower spreading property \eqref{conclusion1}--- when
assumption~\eqref{hyp0} is relaxed.

 \subsection{Proof of \eqref{conclusion2}}\label{ss:with}

We begin by proving some kind of uniform steepness of the
monostable minimal waves, which in turn will easily imply
\eqref{conclusion2}.

\begin{prop}[Steepness of critical waves]\label{steep_min_wave}
Assume that $f$ is of the spatially periodic monostable type, i.e.
$f$ satisfies \eqref{periodicity} and
Assumption~\ref{hyp:monostable}.

Let $u^* (t,x;n) = U^* (x \cdot n - c^* (n) t , x ;n)$ be a family
of increasing in time pulsating
 traveling waves of \eqref{monostable}, with minimal speed $c^*(n)$
 in each direction $n\in \mathbb{S}^{N-1}$, normalized by  $U^* (0,0;n)= \frac{1}{2}$.

Then, the asymptotics $U^*( - \infty, x;n) = 1$, $ U^*(\infty,x;n)
= 0$ are uniform with respect to $n \in \mathbb{S}^{N-1}$.
Moreover, for any $K >0$, we have $\inf_{n \in \mathbb{S}^{N-1}}
\inf_{|z| \leq K} \inf_{ x\in \R^N} -\partial_z U^* (z,x;n)
>0$ and  $\inf_{n \in \mathbb{S}^{N-1}}
\inf_{|z| \leq K} \inf_{ x\in \R^N} U^* (z,x;n)
>0$.
\end{prop}

\begin{rem} [Lack of uniqueness]\label{rem-lack}
Such a family of traveling waves is always known to exist.
However, the uniqueness of the traveling wave with minimal speed
in each direction is not known. We shall prove that any sequence
of increasing in time traveling waves with minimal speed in the
directions $n_k \to n$ converges, up to extraction of a
subsequence, to an increasing in time
 traveling wave with minimal speed in the direction $n$, as we did in the ignition case. The
  proposition then easily follows, but the lack of uniqueness is the reason we state this result in a slightly different way.
\end{rem}
\begin{proof}
Proceeding as explained in the above remark, choose some sequence
$n_k \to n \in \mathbb{S}^{N-1}$. As before, one can extract a
subsequence such that $u^* (\cdot ,\cdot;n_k)$ converges locally
uniformly to a solution $u_\infty$ of \eqref{monostable}. By the
continuity of the speeds $c^* (n)$ with respect to $n$, as proved
in Theorem~\ref{th:continuity_speeds}, the function $u_\infty$
also satisfies
$$\forall l \in \prod_{i=1}^N L_i \mathbb{Z} , \quad u_\infty (t,x) = u _\infty\left( t + \frac{l \cdot n}{c^* (n)},x+l \right).$$
Moreover, it is nondecreasing in time, hence increasing in time by
applying the strong maximum principle to $\partial_t u_\infty$. In
particular, it converges to two spatially periodic stationary solutions as $t \to \pm
\infty$ which, as before and thanks to the monostable assumption,
must be 0 and $1$. As announced, $u_\infty$ is an increasing in
time traveling wave with minimal speed in the direction $n$.
Reasoning by contradiction, it is now straightforward to prove
Proposition~\ref{steep_min_wave}.
\end{proof}

\begin{proof}[Proof of \eqref{conclusion2}] First, from
Proposition~\ref{steep_min_wave} above, and
hypotheses~\eqref{hyp1}---\eqref{hyp0}, one can find some shift
$K_1>0$ large enough so that, for any $n \in \mathbb{S}^{N-1}$,
$u_{0,n} (x) \leq U^* (x \cdot n - K_1,x;n)$. Thus, by the
comparison principle,
$$u_n (t,x) \leq U^* ( x \cdot n- c^* (n) t - K_1 ,x ;n).$$
For any $\alpha >0$ and $\delta >0$, let $\tau$ be such that $U^*
(\alpha \tau - K_1, x;n) \leq \delta$, for all $n \in
\mathbb{S}^{N-1}$ and $x \in \R^N$, which is again made possible
by Proposition~\ref{steep_min_wave}. Then \eqref{conclusion2}
  immediately follows.
\end{proof}

\subsection{Relaxing assumption \eqref{hyp0}}\label{ss:without}

We here consider the case when the family $(u_{0,n})_{n \in
\mathbb{S}^{N-1}}$ does not necessarily satisfy \eqref{hyp0}, but
is only uniformly bounded: there is $M>0$ such that
\begin{equation}\label{bounded}
\forall x\in \R ^N, \forall n\in \mathbb S ^{N-1}, \quad
u_{0,n}(x)\leq M .
\end{equation}
We prove that, in such a situation, the uniform lower and upper
spreading properties \eqref{conclusion1} and \eqref{conclusion2}
remain true if we make the following additional assumptions on the
behavior of $f$, and in particular on its behavior above the
stationary state $p$.
\begin{assumption}[Additional assumptions]\label{ass_add}
\begin{enumerate}[(i)]
\item There is $\phi (t,x)$ a solution of \eqref{monostable} such
that $\phi (0,\cdot) \geq M$, and $\phi (t,x)$ converges uniformly
to $1$ as $t \to +\infty$. \item The steady state 0 of
\eqref{monostable} is linearly unstable with respect to periodic
perturbations. \item There exists some $\rho
>0$ such that $f (x,u)$ is nonincreasing with respect to $u$ in
the set $\R^N \times (1-\rho,1+ \rho)$.
\end{enumerate}
\end{assumption}
The first part of this assumption holds true, for instance, if
$f(x,s) < 0$ for all
 $x \in \R^N$ and $s >1$. As we will see below, the second part can be expressed in
 terms of some principal eigenvalue problem, and holds true as soon as $\partial_u f (x,0)$ is
 positive on a non empty set. The last part is a natural extension of $(iii)$ of Assumption~\ref{hyp:monostable}.\\

Combining \eqref{conclusion1-dessous}, whose proof does not
require assumption \eqref{hyp0}, and a comparison of the solutions $(u_n)_{n\in\mathbb S ^{N-1}}$
with $\phi$ given by the above assumption, it is
 clear that the lower spreading property \eqref{conclusion1} still holds true. In
the sequel, we prove the upper spreading property \eqref{conclusion2}. We
start with the following proposition, whose proof is identical to
that of Proposition~\ref{steep_min_wave} and does not require Assumption \ref{ass_add}.

\begin{prop}[Steepness of noncritical waves]\label{prop}
Assume that $f$ is of the spatially periodic monostable type, i.e.
$f$ satisfies \eqref{periodicity} and
Assumption~\ref{hyp:monostable}.
 
For any $\alpha >0$, let $u_\alpha (t,x;n) = U_\alpha (x \cdot n -
(c^* (n) +\alpha) t,x ;n)$ be a family
of increasing in time pulsating
 traveling waves of \eqref{monostable}, in
direction $n$, with speed $c^* (n)+\alpha$, normalized by
$U_{\alpha} (0,0;n)=\frac{1}{2}$.

Then, the asymptotics $U_\alpha( - \infty, x;n) = 1$, $ U_\alpha
(\infty,x;n) = 0$ are uniform with respect to $n \in
\mathbb{S}^{N-1}$. Moreover, for any $K >0$, we have $\inf_{n \in
\mathbb{S}^{N-1}} \inf_{|z| \leq K} \inf_{ x\in \R^N} -\partial_z
U_\alpha (z,x;n)
>0$ and $\inf_{n \in
\mathbb{S}^{N-1}} \inf_{|z| \leq K} \inf_{ x\in \R^N}
U_\alpha (z,x;n)
>0$
\end{prop}

We now turn to the proof of the upper spreading
property~\eqref{conclusion2}, which relies on the construction of
a suitable family of supersolutions that were already used in
\cite{Ham-Roq} (following an idea of \cite{Fif-Mac}).

\begin{proof}[Proof of \eqref{conclusion2}]Let $\alpha
>0$ and $\delta >0$ be given.  We need to find $\tau
>0$ so that estimate~\eqref{conclusion2} holds for all $t\geq \tau$.

First, we need to introduce some notations, and some well known
results (see \cite{Xin}, \cite{Ber-Ham-02}, \cite{Ham-Roq} among
others). We begin with the principal eigenvalue problem
\begin{equation}\label{principal0}
\left\{
\begin{array}{l}
- L_{0, n , \lambda} \phi_{n,\lambda}  =   \mu_0 (n,\lambda) \phi_{n,\lambda} \ \mbox{ in } \R^N , \vspace{3pt}\\
\phi_{n,\lambda} \mbox{ is periodic}, \; \phi_{n,\lambda} >0,\;
\Vert \phi_{n,\lambda} \Vert
 _\infty=1,
\end{array}
\right.
\end{equation}
where
$$L_{0,n,\lambda} \phi =
\mbox{div} \, (A \nabla \phi) + \lambda^2 (n \cdot A n) \phi -
\lambda (\mbox{div} (An \phi) + n\cdot A\nabla \phi) + q \cdot
\nabla \phi - \lambda (q \cdot n)  \phi + \partial_u f (x,0) \phi
.$$ This arises, similarly as in Section~\ref{s:ignition}, when
looking for moving exponential solutions of the type $e^{-\lambda
(x \cdot n-ct)} \phi_{n,\lambda} (x)$ of the linearized problem
around 0. Such solutions exist if and only if
\begin{equation*}
c \geq c^*_{lin} (n) := \min_{\lambda >0} \frac{- \mu_0
(n,\lambda)}{\lambda},
\end{equation*}
which is well-defined thanks to the linear instability of 0,
which reads as $\mu_0 = \mu_0 (n,0) < 0$. Moreover, it is known that $c^*
(n) \geq c^*_{lin} (n)$ \cite{Ham}. We introduce $\lambda (n)$ the
smallest positive solution of $-\mu_0 (n,\lambda) = \left(c^* (n)
+ \frac{\alpha}{4} \right) \lambda$. It is standard that $\mu_0
(n,\lambda)$ is continuous with respect to $n$ and, as it is known
to be concave, $\lambda (n)$ is also continuous with respect to
$n$. In particular
$$ 0 < \min_{n \in \mathbb{S}^{N-1}} \lambda (n) \leq \max_{n \in \mathbb{S}^{N-1}} \lambda (n) < +\infty.$$

Let some smooth and nonincreasing $\chi (z)$ be such that
$$\chi (z) =
\left\{ \begin{array}{l}
1 \mbox{ if } z < -1,\\
0 \mbox{ if } z > 1,
\end{array}
\right.
$$
and define, for $s\geq 0$ (a shift to be fixed later),
\begin{equation*}
\Phi  (t,x)=\Phi _s(t,x;n) := \chi (\xi _s ) + (1- \chi (\xi _s))
\phi_{n,\lambda (n)} (x) e^{-\lambda (n)\xi _s},
\end{equation*}
where
$$
\xi_s=\xi_s(t,x;n)=x\cdot n- \left(c^*(n)+
\frac{\alpha}{2}\right)t-s.
$$
Note that $\Phi$ is nonnegative and, along with its derivatives,
is bounded uniformly with respect to $n$ and $s$.

Let us now define various positive constants. Choose $0<\eta
<\delta$ small enough so that
\begin{equation}\label{eta-zero}\forall x\in
\R^N , \ \forall 0 \leq u \leq \eta , \quad  | \partial_u f(x,u) -
\partial_u f (x,0) | \leq \frac{\alpha}{4} \min_{n \in
\mathbb{S}^{N-1}} \lambda (n),
\end{equation}
\begin{equation}\label{eta-p}
\forall x\in \R^N , \ \forall 1 - \eta \leq u \leq 1 + \eta ,
\quad  \partial_u f(x,u)  \leq 0.
\end{equation}
Now, by Proposition~\ref{prop}, there is $K>1$ large enough such
that, for all $n \in \mathbb S ^{N-1}$, $x\in \R ^N$,
\begin{equation}\label{K}
\xi > K \Rightarrow 0\leq U_{\frac \alpha 4}(\xi,x;n)\leq \frac
\eta 2,\quad \xi <-K \Rightarrow 1-\frac \eta 2  \leq U_{\frac
\alpha 4}(\xi,x;n)\leq 1.
\end{equation}
Then, by Proposition~\ref{prop} again, we have
\begin{equation}\label{gamma}
\gamma:= \inf_{n \in \mathbb{S}^{N-1}} \inf_{|z| \leq  K , x \in
\R^N} -
\partial_z U_{\frac{\alpha}{4}} (z,x;n) >0.
\end{equation}
Last, we define
$$\epsilon_1 := \frac \eta{2\Vert \Phi \Vert
_\infty} \ , \ \epsilon_2 := \frac {\alpha \gamma}{4(\Vert
\partial _t\Phi \Vert _ \infty+\Vert \mbox{div} (A\nabla  \Phi)
\Vert _\infty + \Vert q \cdot \nabla  \Phi \Vert _\infty+\Vert
\Phi \Vert _\infty \Vert \partial _u f\Vert _{L^\infty(\R^N \times
(0,1 +\frac \eta 2))})}$$ and
\begin{equation}\label{epsilon}
\epsilon:=\min\left(\epsilon _1 , \epsilon _2 \right)>0.
\end{equation}

\medskip

Now, we are going to show that
$$
v(t,x)=v_s(t,x;n):=U_{\frac{\alpha}{4}} \left(x \cdot n -
\left(c^* (n) +\frac{\alpha}{2}\right) t-s,x;n \right) +
\epsilon\Phi (t,x)=U_{\frac{\alpha}{4}} (\xi _s,x;n) +
\epsilon\Phi (t,x)
$$
is a supersolution of the monostable equation~\eqref{monostable}.
Straightforward computations and the mean value Theorem yield
\begin{eqnarray*}\mathcal L[v](t,x)&:=& \partial _t v(t,x)- \mbox{div} (A (x) \nabla v (t,x))  - q (x) \cdot \nabla v(t,x) - f( x,v(t,x))\\
&=&\epsilon \left[\partial _t \Phi(t,x)- \mbox{div} (A (x) \nabla
\Phi(t,x)) - q(x) \cdot \nabla \Phi (t,x) - \Phi(t,x)\partial _u f(x, \theta (t,x)) \right] \\
&& \qquad -\frac \alpha 4
\partial _z U_{\frac \alpha 4}(\xi _s,x;n),
\end{eqnarray*}
for some
$$
U_{\frac \alpha 4}(\xi _s,x;n)\leq \theta (t,x) \leq U_{\frac
\alpha 4}(\xi _s,x;n)+\epsilon \Phi(t,x).
$$
We distinguish three regions, depending on the values of $\xi _s$.

First, if $|\xi _s|\leq K$, the nonnegativity of $\mathcal
L[v](t,x)$ is obtained thanks to $-\frac \alpha 4
\partial _z U_{\frac \alpha 4}(\xi _s,x;n)\geq \frac \alpha 4 \gamma$
by \eqref{gamma} and the definition of~$\epsilon$ in
\eqref{epsilon}.

Next, if $\xi _s
>K$, then $\Phi(t,x)$ reduces to $\phi_{n,\lambda (n)} (x) e^{-\lambda
(n) \xi _s}$ and, dropping $-\frac \alpha 4
\partial _z U_{\frac \alpha 4}(\xi _s,x;n)$ which is positive, we
arrive at
\begin{eqnarray*}\frac 1 \epsilon\mathcal L[v](t,x)&\geq& \left[
\lambda(n)\left( c^*(n)+\frac{\alpha}{2}\right)+\mu
_0(n,\lambda(n))+\partial _u f(x,0)-\partial _u f(x,\theta(t,x))
\right] \phi_{n,\lambda (n)}
(x)e^{-\lambda (n) \xi _s}\\
&\geq &\left(\frac \alpha 4 \lambda (n)+\partial _u
f(x,0)-\partial _u f(x,\theta(t,x)) \right)\phi_{n,\lambda (n)}
(x)e^{-\lambda (n) \xi _s}.
\end{eqnarray*}
But, when $\xi _s >K$, \eqref{K} and $\epsilon\leq \epsilon _1$
imply $0\leq \theta (t,x)\leq \eta $, and the nonnegativity of
$\mathcal L[v](t,x)$ is obtained thanks to \eqref{eta-zero}.

Last, we consider the case where $\xi _s<-K$, so that $\Phi(t,x)$
reduces to $1$. Hence
$$\frac{1}{\epsilon} \mathcal L[v](t,x)\geq -\partial _u f(x,\theta(t,x)).
$$
But, when $\xi _s<-K$, \eqref{K} and $\epsilon\leq \epsilon _1$
imply $1-\eta\leq\theta(t,x)\leq 1+\eta$, and the nonnegativity of
$\mathcal L[v](t,x)$ is obtained thanks to \eqref{eta-p}. Hence, $
v_s(t,x;n)$ is a supersolution of \eqref{monostable}.

\medskip

Thanks to \eqref{bounded}, we get by the comparison principle that, for all $ n \in \mathbb{S}^{N-1}$, all $t\geq 0$, all $x \in \R^N$, $ u_n (t,x) \leq \phi (t,x)$,
 where $\phi$ is
given by Assumption~\ref{ass_add}. 
Now choose $T >0$ such that $\phi (T,x) \leq 1 +
\frac{\epsilon}{2}$, and get that
\begin{equation}\label{end1}
\forall n \in \mathbb{S}^{N-1}, \  \forall x \in \R^N , \quad  u_n (T,x) \leq  1 +
\frac{\epsilon}{2}.
\end{equation}

Using the comparison principle and a computation identical to that of \eqref{gggg}, we get that, for any large $\lambda >0$, there is $C>0$ --- independent on $n$ thanks to \eqref{hyp1} and
\eqref{bounded}--- such that
$$
\forall n \in \mathbb{S}^{N-1}, \ \forall t \geq 0 , \ \forall x \in \R^N , \quad u_n (t,x) \leq Ce^{-\lambda (x\cdot n-2a_1\lambda t)}.
$$
In particular $u_n (T,\cdot)$ decays faster than any
exponential as $x\cdot n \to +\infty$, namely
\begin{equation}\label{faster}
\forall \lambda >0 , \quad u_n (T,x) e^{\lambda x\cdot n}
 \to 0 \ \mbox{ as } x \cdot n \to +\infty, \ \mbox{ uniformly w.r.t. } n \in \mathbb{S}^{N-1}.
 \end{equation}

Observe that, for all $s\geq 0$, 
\begin{equation}\label{droite-s}
\forall n\in \mathbb S  ^{N-1},\ \forall x\cdot n\geq \left(c^{*}(n)+\frac{\alpha}{2}\right)T+s+1, \ v_s(T,x;n)\geq \epsilon \phi _{n,\lambda(n)}(x)e^{-\lambda(n)x\cdot n}\geq \epsilon \gamma e^{-\lambda_{max}x\cdot n},
\end{equation}
where $\gamma:=\min_{n \in
\mathbb{S}^{N-1}} \min_{ x\in \R^N} \phi_{n,\lambda (n)} (x)
>0$ and $\lambda _{max}:=\max_{n\in \mathbb S ^{N-1}} \lambda(n)<\infty$ (recall that $n\mapsto \lambda (n)$ is
continuous and so is $(n,\lambda)\mapsto \phi_{n,\lambda}$). Now, select $A>1$ large enough so that, for all $s\geq 0$,
\begin{equation}\label{gauche-s}
\forall n\in \mathbb S  ^{N-1},\ \forall x\cdot n\leq \left(c^{*}(n)+\frac{\alpha}{2}\right)T+s-A, \ v_s(T,x;n)\geq 1+\frac \epsilon 2,
\end{equation}
which is possible thanks to Proposition~\ref{prop}, and more
precisely the uniform with respect to~$n$ asymptotics of
$U_{\frac{\alpha}{4}} (z,x;n)$ as $z \to -\infty$. Proposition~\ref{prop} also enables to define
$$
\kappa:=\inf_{n\in \mathbb S ^{N-1}} \inf _{-A\leq z\leq 1} \inf _{x\in \R ^{N}}U_{\frac{\alpha}{4}} (z,x;n)>0,
$$
so that, for all $s\geq 0$,
\begin{equation}\label{milieu-s}
\forall n\in \mathbb S  ^{N-1},\ \forall \left(c^{*}(n)+\frac{\alpha}{2}\right)T+s-A\leq x\cdot n\leq \left(c^{*}(n)+\frac{\alpha}{2}\right)T+s+1, \ v_s(T,x;n)\geq \kappa.
\end{equation}

In view of \eqref{faster}, we can now select a large enough shift $s_0>A$ so that
$$
\forall n\in \mathbb S  ^{N-1},\ \forall x\cdot n\geq \left(c^{*}(n)+\frac{\alpha}{2}\right)T+s_0-A, \ u_n(T,x)\leq \min \{ \epsilon  \gamma , \kappa \} e^{-\lambda_{max}x\cdot n}.
$$
Combining this with \eqref{end1}, \eqref{droite-s}, \eqref{gauche-s}, \eqref{milieu-s}, we have that, for all $n \in \mathbb{S}^{N-1}$ and $x \in \R^N$,
\begin{equation*}
u_n(T,x)\leq v_{s_0}(T,x;n).
\end{equation*}

\medskip

Then, by the comparison principle, for all $t\geq T$, $x\in \R
^N$, $n\in\mathbb S ^{N-1}$,
$$
0\leq u_n(t,x)\leq  U_{\frac \alpha 4} \left( x\cdot n- \left(
c^*(n)+ \frac{\alpha}{2} \right)t-s_0,x;n \right)+\epsilon
\Phi(t,x).
$$
Hence, when $x\cdot n\geq (c^*(n)+\alpha )t$, we have, since
$\epsilon \Vert \Phi \Vert _\infty \leq \frac \eta 2 \leq \frac
\delta 2$, that
$$
0\leq u_n(t,x)\leq  U_{\frac \alpha 4}\left( \frac{\alpha}{2}
t-s_0,x;n\right)+\frac \delta 2 \leq \delta,
$$
as soon as $t\geq \tau$, where $\tau>0$ is large enough (again
independently on $n$ by Proposition~\ref{prop}). This proves
\eqref{conclusion2}.\end{proof}

\section{The uniform spreading: the ignition case}\label{s:spread_ignition}

For the sake of completeness, we give here the main steps to prove
 Theorem~\ref{th:unif_spreading} in the (simpler) ignition
 case. We will see that it  follows from the continuity of
 ignition waves, Theorem~\ref{th:continuity_igwaves},
  together with the standard idea explained in Remark~\ref{rem_bug}. We will briefly sketch at
  the end of this section how the hypothesis \eqref{hyp0} can again be relaxed.

  \begin{proof}[Proof of Theorem \ref{th:unif_spreading} in the
  ignition case]
First, the proof of the  uniform upper
spreading~\eqref{conclusion2} is  the same as that of subsection
\ref{ss:with}
 in the monostable case, using Theorem~\ref{th:continuity_igwaves} (continuity of ignition waves) instead of Proposition~\ref{steep_min_wave} (steepness of critical
 waves).

\medskip
Let us now prove the uniform lower spreading~\eqref{conclusion1}.
Let  $\alpha >0$ and $\delta >0$ be given.
 We may reduce $\delta$ without loss of generality, and assume that $\delta < \rho$ where $\rho$ is given by part
 $(iii)$ of Assumption~\ref{hyp:ignition}. Using the same
 arguments as in the proof of \eqref{claim:proof_uniformspread1},
 we get the existence of some time $t_\delta>0$ such that
$$
u_n (t_\delta,x) \geq 1 - \frac{\delta}{2},
$$
for all $n \in \mathbb{S}^{N-1}$ and all $x$ such that $x \cdot n
\leq - K$. Now let, as usual, $U^* (x\cdot n - c^* (n) t,x;n)$ be
the unique ignition pulsating wave in the direction~$n$,
normalized by $\min _{x\in \R ^N} U^* (0,x;n) =
\frac{1+\theta}{2}.$ Thanks to the inequality above and the
continuity of the mapping $n \mapsto U^* (\cdot , \cdot;n)$ with
respect to the uniform topology, it is clear that there exists
some shift $Z>0$ such that, for all $n\in \mathbb{S}^{N-1}$,
$$
U^* (x \cdot n +Z ,x ;n) - \frac{\delta}{2} \leq  u_n (t_\delta
,x), \quad \forall x\in \R ^N.
$$

We then check that $\underline{u} (t,x) := U^* (x \cdot n + Z -
(c^* (n) - \frac{\alpha}{2}) t,x;n) - \frac{\delta}{2}$ is a
subsolution of \eqref{monostable}. Indeed,
\begin{equation*}
 \partial_t \underline{u}- \mbox{div} \, (A(x) \nabla \underline{u}) - q(x) \cdot \nabla \underline{u} - f(x,\underline{u}) =  \frac{\alpha}{2} \partial_z U^* + f(x,U^*) - f(x,\underline{u}).
\end{equation*}
Assume first that $\underline{u} \leq \theta - \frac{\delta}{2}$.
Then $f(x,\underline{u})= f(x,U^* (x\cdot n + Z - (c^* (n) -
\frac{\alpha}{2} )t,x;n)=0$, which together with the monotonicity
of $U^*$ with respect to its first variable, gives the wanted
inequality. Assume then that $\underline{u} \geq 1 - \rho$. Then,
by the monotonicity of $f$ with respect to $u$ in the range
$[1-\rho,1]$, we again obtain the wanted inequality. 

It remains to
prove that $\underline{u}$ is a subsolution when $\theta -
\frac{\delta}{2} \leq \underline{u} \leq 1 - \rho$ or,
 equivalently, when $\theta \leq U^* (x\cdot n - (c^* (n)- \frac{\alpha}{2})t+Z,x;n) \leq 1 - \rho + \frac{\delta}{2}$.
 Recall first that $1 - \rho + \frac{\delta}{2}< 1 - \frac{\rho}{2}$. Using again the continuity of the ignition wave
 with respect to the direction, we have that there exists some~$R>0$ such
 that, for all $n\in \mathbb S ^{N-1}$,
$$z \geq R \ \Rightarrow \ U^* (z + Z ,x;n) < \theta, \quad z \leq -R \ \Rightarrow \  U^* (z+Z,x;n) > 1 - \frac{\rho}{2},$$
and, furthermore,
$$\max_{n \in \mathbb{S}^{N-1}} \max_{|z| \leq R} \  \max_{x \in \R^N} \  \partial_z U^* (z+Z,x;n) <0.$$
Up to reducing $\delta$ again, we may assume that
$$\max_{n \in \mathbb{S}^{N-1}} \max_{|z| \leq R} \  \max_{x \in \R^N} \  \partial_z U^* (z+Z,x;n) < - \frac{M\delta}{\alpha},$$
where $M$ is a Lipschitz constant of $f$. Therefore, when $\theta
- \frac{\delta}{2} \leq \underline{u} \leq 1 - \rho$, then $|x
\cdot n - (c^* (n) - \frac{\alpha}{2} )t | \leq R$ and
\begin{eqnarray*}
&& \partial_t \underline{u}- \mbox{div} \, (A(x) \nabla \underline{u}) - q(x) \cdot \nabla \underline{u} - f(x,\underline{u}) \vspace{3pt}\\
& = &  \frac{\alpha}{2} \partial_z U^* + f(x,U^*) - f(x,\underline{u}) \vspace{3pt}\\
& \leq & \frac{\alpha}{2} \partial_z U^* +  M \frac{\delta}{2}
\leq 0,
\end{eqnarray*}
that is the wanted inequality.

We can therefore apply the comparison principle and conclude that
$$U^* \left(x \cdot n + Z - (c^* (n) - \frac{\alpha}{2})t ,x;n \right) - \frac{\delta}{2}\leq u_n (t_\delta +t,x),$$
for all $x \in \R^N$ and $t \geq 0$.

Noting that there exists some other shift $Z'>0$ such that, for
all $ n \in \mathbb{S}^{N-1}$,
$$z \leq -Z ' \ \Rightarrow  \ U^* (z,x;n) - \frac{\delta}{2} \geq  1 -\delta,$$
we get the uniform lower spreading \eqref{conclusion1} as in the
end of Section \ref{s:lower-spreading}.
\end{proof}

\begin{proof}[Relaxing hypothesis \eqref{hyp0}]
In order to relax  \eqref{hyp0}, assume now that $f$ satisfies
Assumption~\ref{hyp:ignition} and parts $(i)$ and $(iii)$ of
Assumption~\ref{ass_add}. As above, one can then show that $U^* (x
\cdot n - (c^* (n) + \frac{\alpha}{2})t,x;n) + \frac{\delta}{2}$
is a supersolution of \eqref{monostable}. Then, as in
Section~\ref{ss:without}, one can find some time $T$ and some
shift $s_0$ such that, for all $n$, the solution $u_n (T,x)$ lies
below $U^* (x \cdot n + s_0 - ( c^* (n)+ \frac{\alpha}{2})T,x;n) +
\frac{\delta}{2}$ in the whole space. It is then straightforward
to obtain the wanted uniform upper spreading~\eqref{conclusion2}.
\end{proof}

\bigskip

\noindent \textbf{Acknowledgement.} M. A. was supported by the
French {\it Agence Nationale de la Recherche} within the project
IDEE (ANR-2010-0112-01). T. G. was supported by the French {\it Agence Nationale de la Recherche} within the project NONLOCAL (ANR-14-CE25-0013). Both authors are grateful to Professor
Hiroshi Matano for great hospitality in the University of Tokyo,
where this work was initiated.


\begin{thebibliography}{ABCD}





 \bibitem{A-Gil} M.~Alfaro and T. Giletti, {\it Asymptotic analysis of a monostable equation in
periodic media}, in preparation.

\bibitem{Aro-Wei1} D. G. Aronson and H. F. Weinberger, {\it Nonlinear diffusion in population genetics, combustion,
 and nerve pulse propagation},  Partial differential equations and related topics (Program,
 Tulane Univ., New Orleans, La., 1974), 5--49. Lecture Notes in Math., Vol. 446, Springer, Berlin, 1975.

\bibitem{Aro-Wei2} D. G. Aronson and H. F. Weinberger, {\it Multidimensional
nonlinear diffusion arising in population genetics}, Adv. in Math.
{\bf 30}  (1978), no. 1, 33--76.



\bibitem{Ber-Ham-02} H. Berestycki and F. Hamel,
{\it Front propagation in periodic excitable media}, Comm. Pure
Appl. Math. {\bf 55} (2002), no. 8, 949--1032.

\bibitem{berestycki-hamel-cpam2} H. Berestycki and F. Hamel, \emph{Generalized transition waves and their properties},
Comm. Pure Appl. Math. {\bf 65}, (2012), no. 5, 592--648.


\bibitem{Ber-Ham-Nad08} H. Berestycki, F. Hamel and G. Nadin, {\it Asymptotic spreading in heterogeneous diffusive excitable
              media}, J. Funct. Anal. {\bf 255} (2008), 2146--2189.

\bibitem{Ber-Ham-Roq1} H. Berestycki, F. Hamel and L. Roques, {\it Analysis of the periodically fragmented environment model.
I. Species persistence}, J. Math. Biol. {\bf 51} (2005), 5--113.

\bibitem{Ber-Ham-Roq2} H. Berestycki, F. Hamel and L. Roques, {\it Analysis of the periodically fragmented environment model.
II. Biological invasions and pulsating traveling fronts}, J. Math.
Pures Appl. {\bf 84} (2005), 1101--1146.

%

\bibitem{Ber-Nic-Sch} H. Berestycki, B. Nicolaenko and B.
Scheurer, {\it Traveling wave solutions to combustion models and
their singular limits}, SIAM J. Math. Anal. {\bf 16} (1985), no.
6, 1207--1242.

\bibitem{berestycki-nirenberg}
H.~Berestycki and  L. ~Nirenberg, \emph{Travelling fronts in
cylinders}, Ann. Inst. H. Poincar\'e Anal. Non Lin\'eaire {\bf 9}
(1992), no. 5, 497--572.






\bibitem{Fif-Mac} P. C. Fife and J. B. McLeod,
{\it The approach of solutions of nonlinear diffusion equations to
travelling front solutions}, Arch. Rational Mech. Anal. {\bf 65}
(1977), 335--361.

\bibitem{Fish} R. A. Fisher, {\it The wave of advance of advantageous genes},
Ann. of Eugenics {\bf 7} (1937), 355--369.




\bibitem{Ham} F. Hamel, {\it Qualitative properties of monostable pulsating fronts:
exponential decay and monotonicity}, J. Math. Pures Appl. {\bf 89}
(2008), 355--399.

\bibitem{Ham-Roq} F. Hamel and L. Roques, {\it Uniqueness and stability properties of monostable pulsating
fronts}, J. Eur. Math. Soc. {\bf 13} (2011), 345--390.

\bibitem{HZ} W. Hudson and B. Zinner, \emph{Existence of traveling waves for
reaction diffusion equations of {F}isher type in periodic media},
Boundary value problems for functional-differential equations,
187--199, World Sci. Publ., River Edge, NJ, 1995.

\bibitem{Kan1} Ja. I. Kanel, {\it Stabilization of solutions of the Cauchy problem for equations encountered in
combustion theory},  (Russian) Mat. Sb. {\bf 59} (1962), 245--288.

\bibitem{Kol-Pet-Pis} A. N. Kolmogorov, I. G. Petrovsky and N. S. Piskunov, {\it Etude de
l'\'equation de la diffusion avec croissance de la quantit\'e de
mati\`ere et son application \`a un probl\`eme biologique},
Bulletin Universit\'e d'Etat  Moscou, Bjul. Moskowskogo Gos.
Univ., 1937, 1--26.

\bibitem{Lie} G. M. Lieberman, {\it Second Order Parabolic Differential Equations}, World Scientific Publishing
Co. Inc., River Edge, NJ, 1996.




\bibitem{Nad-09} G. Nadin, {\it Traveling fronts in space-time periodic media}, J. Math. Pures
Appl. (9) {\bf 92} (2009), no. 3, 232--262.

\bibitem{Nolen-Ryzhik} J. Nolen and L. Ryzhik, {\em Traveling waves in a one-dimensional
 heterogeneous medium}, Ann. Inst. H. Poincar\'e Anal. Non Lin\'eaire {\bf 26} (2009), no. 3, 1021--1047.


\bibitem{Qui-Sou} P. Quittner and P. Souplet, {\it Superlinear Parabolic
Problems}, Birkh\"auser Verlag, Basel, 2007.

\bibitem{Shi-Kaw} N. Shigesada and K. Kawasaki, {\it Biological
Invasion: Theory and Practise}, Oxford University Press, 1997.


\bibitem{Volpert-Volpert-Volpert} A.\ Volpert, V. Volpert, V. Volpert,
{\it Travelling Wave Solutions of Parabolic Systems}, Translations
of Mathematical Monographs, vol.\ 140, AMS\ Providence, RI, 1994.

\bibitem{Wein02} H.~Weinberger, {\it On spreading speed and travelling waves for growth and migration},
J. Math. Biol. {\bf 45} (2002), 511--548.

\bibitem{Xin} J. Xin, {\it Existence of planar flame fronts in
convective-diffusive periodic media}, Arch. Ration. Mech. Anal.
{\bf 121} (1992), 205--233.

\bibitem{Xin3} J. Xin, \emph{Front propagation in heterogeneous media}, SIAM Rev. {\bf 42} (2000), no. 2, 161--230.

\bibitem{Zlatos} A. Zlato{\v{s}}, \emph{Generalized traveling waves in disordered media: Existence, uniqueness, and stability}, Arch. Ration. Mech. Anal. 208 (2013), 447-480.

\end{thebibliography}
\end{document}